\documentclass{amsart}

\usepackage{amsmath,amssymb,amsthm}

\usepackage{graphicx,tikz, caption,bigints}

\usepackage{subfig}

\newtheorem{theorem}{Theorem}
\newtheorem*{thm}{Theorem}

\newtheorem{corollary}[theorem]{Corollary}

\newtheorem{lemma}[theorem]{Lemma}

\theoremstyle{definition}

\newtheorem{example}{Example}

\theoremstyle{remark}

\renewcommand{\div}{\operatorname{div}}

 \DeclareMathOperator\curl{curl}

\title[]{On localization of eigenfunctions \\ of the magnetic Laplacian}

\author[]{Jeffrey S. Ovall, Hadrian Quan, \\Robyn Reid \and Stefan Steinerberger}
\address{Jeffrey S. Ovall,
  Fariborz Maseeh Department of Mathematics and Statistics,
  Portland State University,
  Portland, OR 97201}
\email{jovall@pdx.edu}

\address{Hadrian Quan,
	Department of Mathematics,
        University of Washington,
	Seattle, WA 98195}
\email{hadrianq@uw.edu}

\address{Robyn Reid,
	Fariborz Maseeh Department of Mathematics and Statistics,
        Portland State University,
        Portland, OR 97201}
\email{reid3@pdx.edu}

\address{Stefan Steinerberger,
	Department of Mathematics,
        University of Washington,
	Seattle, WA 98195}
\email{steinerb@uw.edu}

\keywords{Localization, Eigenfunction, Schr{\"o}dinger Operator, Regularization.}
\subjclass[2010]{35J10, 65N25 (primary), 82B44 (secondary)} 
\thanks{This work was partially supported by the
	National Science Foundation through
	NSF grants DMS-2012285, DMS-2123224 and
	RTG grant DMS-2136228.}

\begin{document}

\begin{abstract}
Let $\Omega \subset \mathbb{R}^d$ and consider the magnetic Laplace operator given by $ H(A) = \left(- i\nabla - A(x)\right)^2$, where $A:\Omega \rightarrow \mathbb{R}^d$, subject to Dirichlet boundary conditions.  For certain vector fields $A$, this operator can have eigenfunctions, $H(A) \psi = \lambda \psi$, that are highly localized in a small region of $\Omega$. The main goal of this paper is to show that if $|\psi|$ assumes its maximum at $x_0 \in \Omega$, then $A$ behaves `almost' like a conservative vector field in a $1/\sqrt{\lambda}-$neighborhood of $x_0$ in a precise sense. In particular, we expect localization in regions where $\left|\curl A \right|$ is small. The result is illustrated with numerical examples.
\end{abstract}

\maketitle

\section{Introduction and Results}
\subsection{Introduction}
Given a Schr\"odinger operator $H$, this paper concerns mechanisms by which one can provide an \textit{a priori} prediction, based directly on $H$, of where its eigenfunctions may localize.  Suppose that $\Omega \subset \mathbb{R}^d$ is some bounded domain with smooth boundary (this assumption is purely for convenience) and assume that $V: \Omega \rightarrow \mathbb{R}_{\geq 0}$ is some potential, and consider the operator
$$ H \psi = - \Delta \psi + V \psi~.$$
It is known that  if $V$ changes slowly over large regions, then eigenfunctions of $H$ having eigenvalues near the bottom of the spectrum
should localize near the local minima of $V$. However, if $V$ oscillates extremely rapidly, this heuristic fails. Filoche and Mayboroda \cite{fil} proposed to instead solve the PDE
$$ H u = 1$$
and proved that the solution of this equation satisfies
\begin{equation} \label{eq:landscape}
 \frac{|\psi(x)|}{\| \psi\|_{L^{\infty}}} \leq \lambda \, u(x)
\end{equation}
for any eigenpair $(\lambda, \psi)$ of $H$.
In particular, an eigenfunction associated with $\lambda$ can only localize in the region $\left\{x \in \Omega: \lambda \, u(x) \geq 1 \right\}$. This has inspired a lot of subsequent research 
\cite{alt, altmann, arnold0, arnold, arnold2, har, lierl, rachh, steini} and raised the general question of how to predict localization in rough environments. There are now a variety of methods available, we refer to Altmann-Peterseim \cite{altmann}, the random inner product method \cite{lu} (see also Nenashev-Baranovskii-Meerholz-Gebhard \cite{phys0, phys}), the local landscape function \cite{li2, steini3}, operator perturbation techniques~\cite{ovall}, approaches motivated by operator theory \cite{mugnolo} and others.
Another natural problem of physical relevance is to consider general magnetic Schr\"odinger operator
$$ H(A) \psi =  \left(- i\nabla - A(x)\right)^2 \psi + V \psi,$$
where $A: \Omega \rightarrow \mathbb{R}^d$ is a vector field. This case is much less understood; early theoretical results are due to Z. Shen \cite{shen}. A variant of the Filoche-Mayboroda landscape was proposed by Poggi \cite{poggi}. Hoskins, Quan and Steinerberger \cite{hoskins} proved that the inequality \eqref{eq:landscape} remains
true -- in particular, one can simply go ahead and ignore the vector field $A$ completely. This seems to lead to surprisingly good predictions in a variety of cases. It seems as if, in practice, the potential $V$ has a lot more impact on localization than the vector field $A$. This, naturally, is in need of further clarification: if $\|A\|$ is very large or $V$ is very small, then one would expect the vector field to come into play. This was partially the motivation for the work that lead to the result presented in this paper: in practice, the localization behavior of $ \left(- i\nabla - A(x)\right)^2 + V$ is bound to be an interplay between the operators $-\Delta +V$ and  $ \left(- i\nabla - A(x)\right)^2$,  presumably with one dominating the other in the generic case. The localization behavior of $-\Delta +V$ is, at this point, reasonably well understood, so we focus on the magnetic Laplacian $ \left(- i\nabla - A(x)\right)^2$. 

\subsection{The Problem} We study the magnetic Laplacian
$$ H(A) = \left(- i\nabla - A(x)\right)^2$$
and the behavior of its eigenfunctions.
 The setting is as follows: we assume $\Omega \subset \mathbb{R}^d$ to be a bounded domain with smooth boundary and we assume $A: \Omega \rightarrow \mathbb{R}^d$ to be a differentiable vector field. The assumptions on the regularity of $\partial\Omega$ could be relaxed or dropped.  However, we are mainly interested in localization caused by $A$ as opposed to localization caused by boundary effects: boundary localization is a problem interesting in its own right \cite{felix, jones}, albeit of a completely different nature. Our main problem is now easy to state.
 
 \begin{quote}
   \textbf{Problem.} Predict whether and where eigenvectors of $H(A)$ may localize based on $A$ and eigenvalues $\lambda$ of $H(A)$. 
 \end{quote}

Before describing our main result, we start with a couple of observations. The first is that the problem is, in some sense, much less constrained than, say, $-\Delta + V$. For the pure Schr\"odinger case, we have (assuming the eigenfunction $\psi$ to be normalized in $L^2$) that
$$ \lambda = \left\langle (-\Delta + V)\psi, \psi \right\rangle = \int_{\Omega} |\nabla \psi|^2 dx + \int_{\Omega} V(x) \psi(x)^2dx~,$$
which immediately implies that $\psi$ is primarily localized in $\left\{x \in \Omega: V(x) \leq \lambda \right\}$. This is very well understood: we refer to the celebrated Agmon Theorem \cite{agmon1} and some of its recent variations \cite{filoche, keller, steini4, steini5}. 
In the magnetic case, we have, again assuming $\| \psi\|_{L^2} = 1$, the identity
$$ \lambda = \left\langle  \left(- i\nabla - A(x)\right)^2\psi, \psi \right\rangle = \int_{\Omega} \left| \left(- i\nabla - A(x)\right) \psi(x) \right|^2 dx~,$$
which does not lead to any natural a priori predictions of where localization could occur. (There is an interesting special case in $d=2$: if the magnetic field $B(x)$ is non-negative, then an analogue of this inequality remains true; we are grateful to Bernard Helffer for this remark).

\begin{figure}[h!]
	\subfloat[{$A$.}]
	{\includegraphics[width=0.48\textwidth]{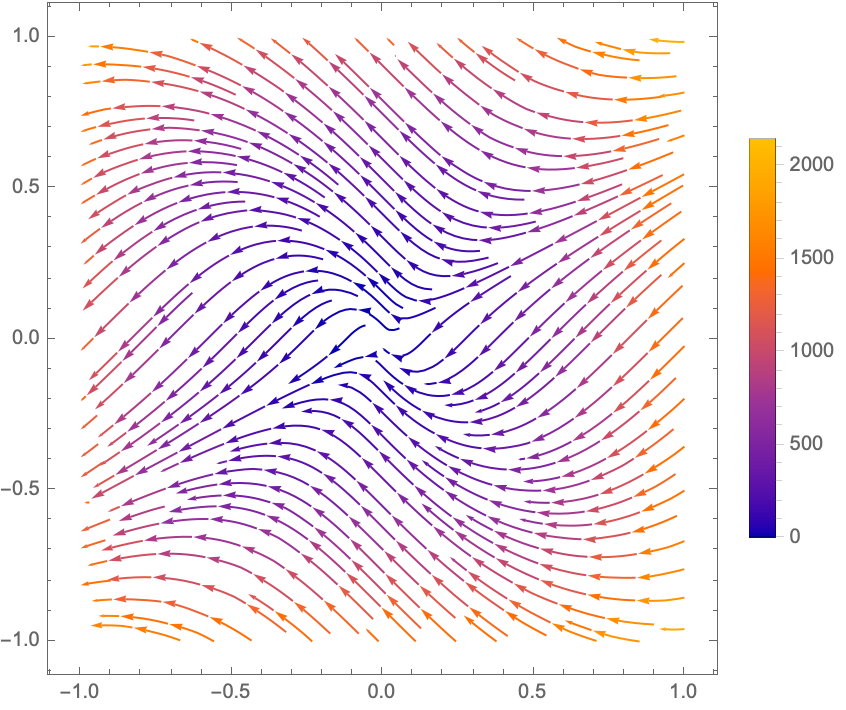}}
        \quad
 	\subfloat[{$|\curl A|$.}]
	{\includegraphics[width=0.48\textwidth]{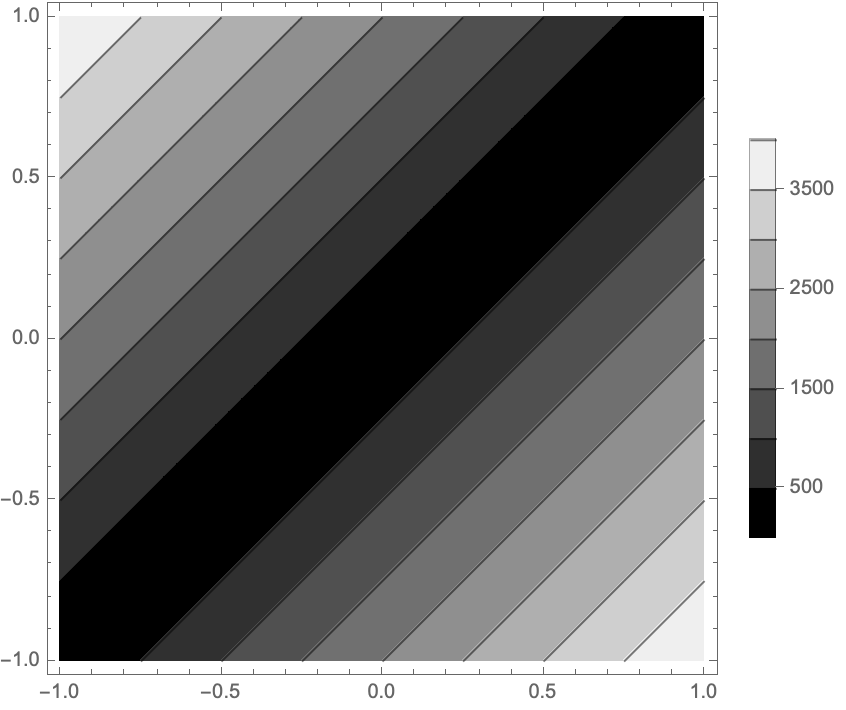}}
        \\
        \subfloat[{$|\psi|$, $\lambda\approx 120.52$.}]
	{\includegraphics[width=0.48\textwidth]{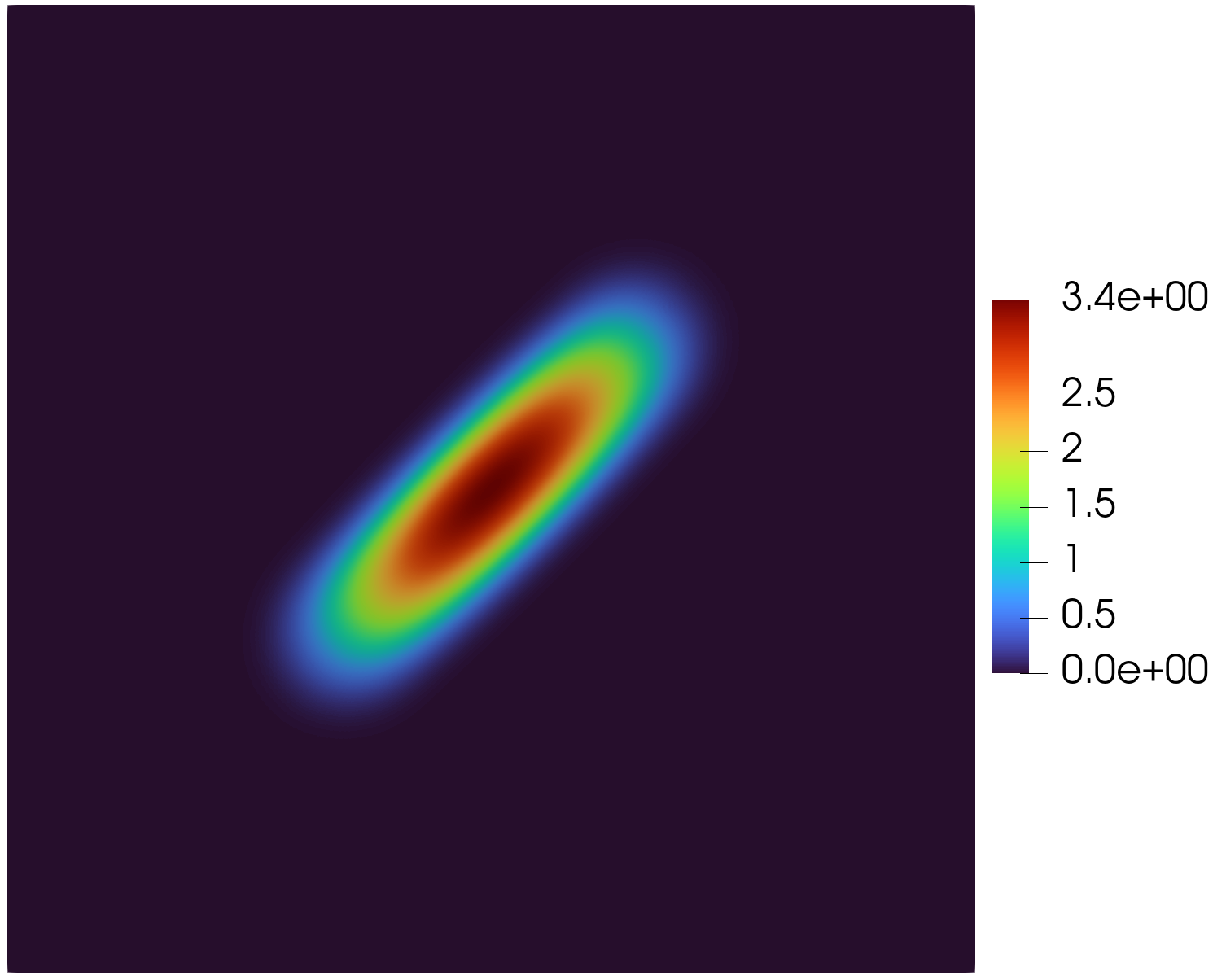}}
        \quad
 	\subfloat[{$|\psi|$, $\lambda\approx 139.80$.}]
	{\includegraphics[width=0.48\textwidth]{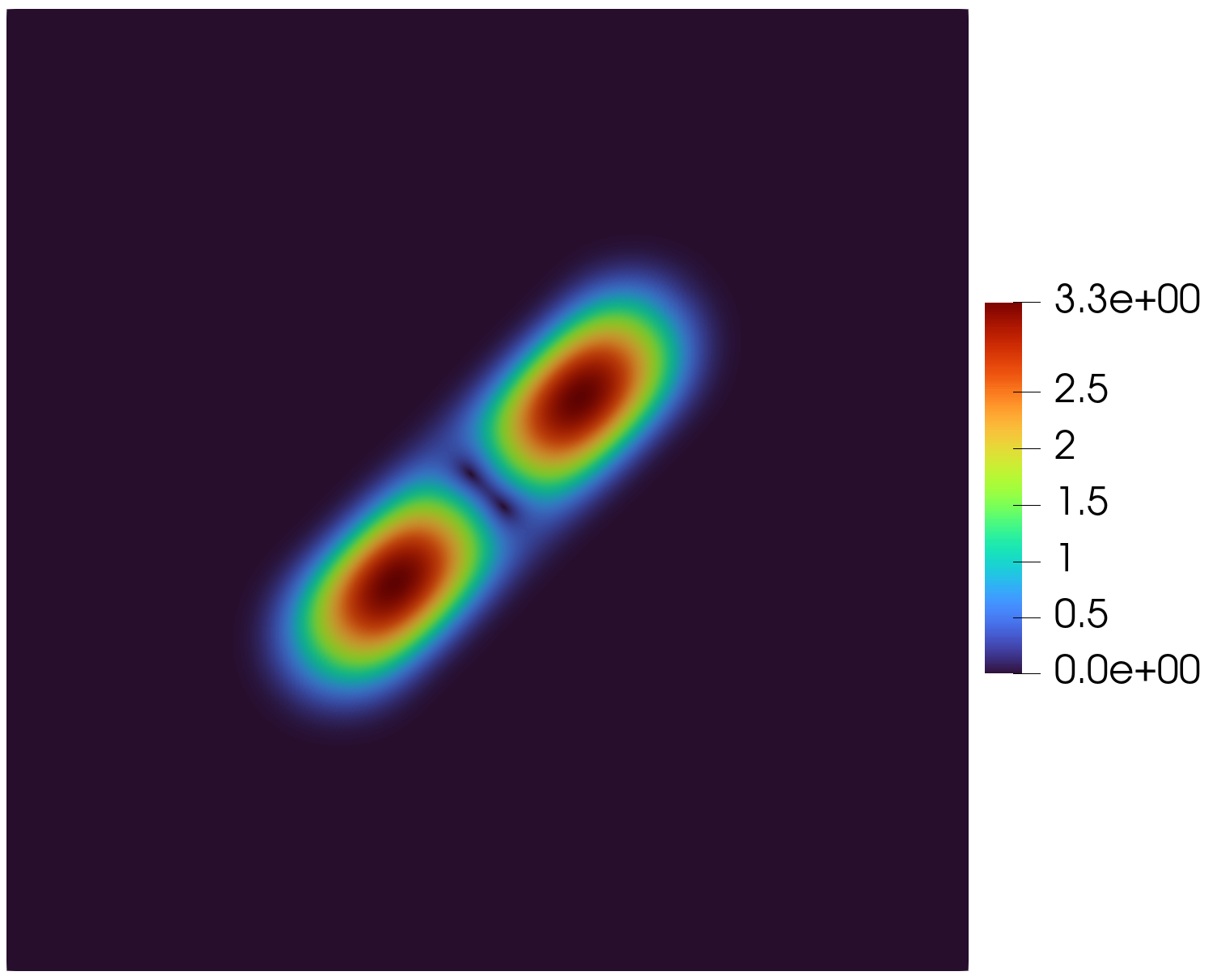}} 
        \\
        \subfloat[{$|\psi|$, $\lambda\approx 170.62$.}]
	{\includegraphics[width=0.48\textwidth]{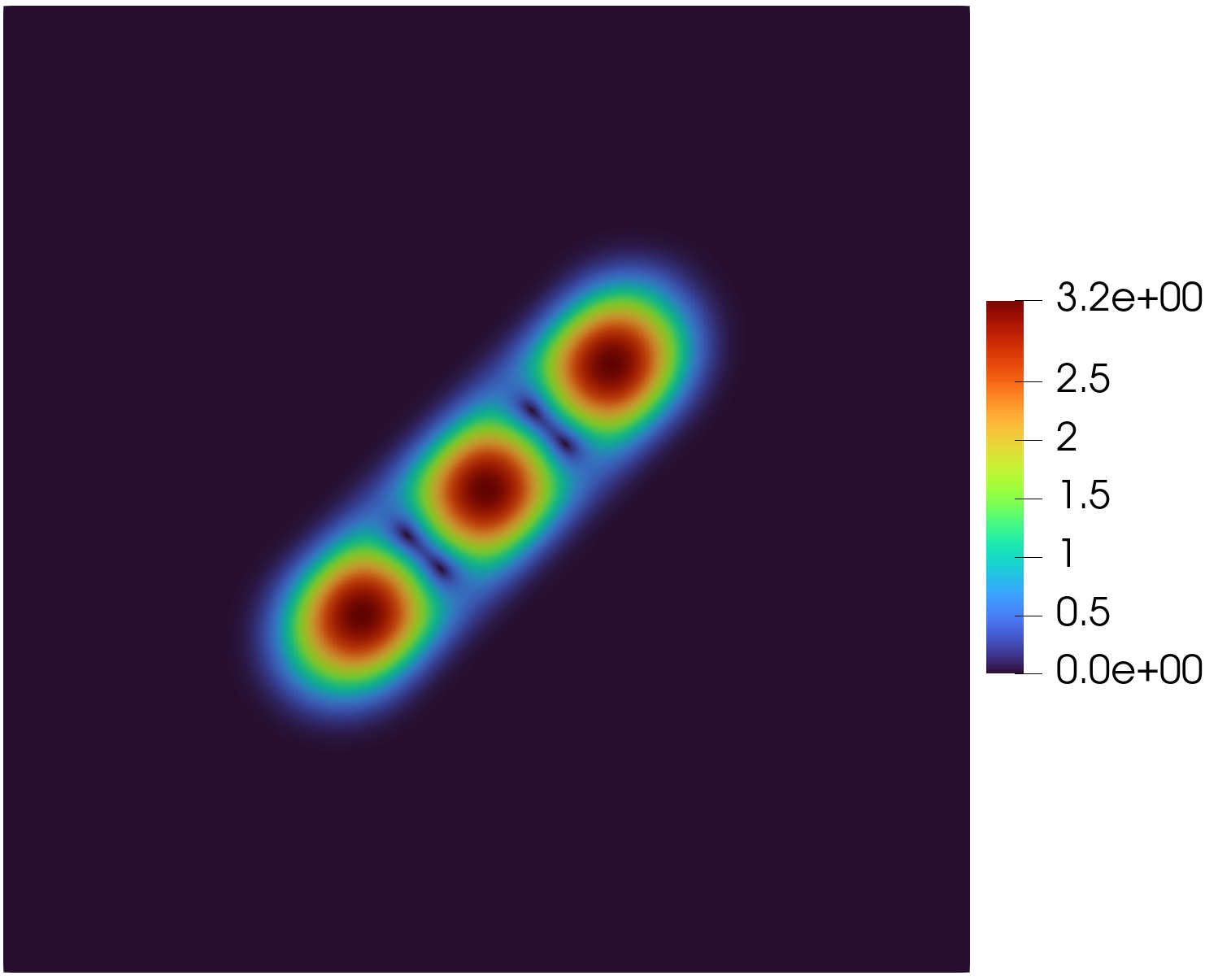}}
        \quad
 	\subfloat[{$|\psi|$, $\lambda\approx 211.07$.}]
	{\includegraphics[width=0.48\textwidth]{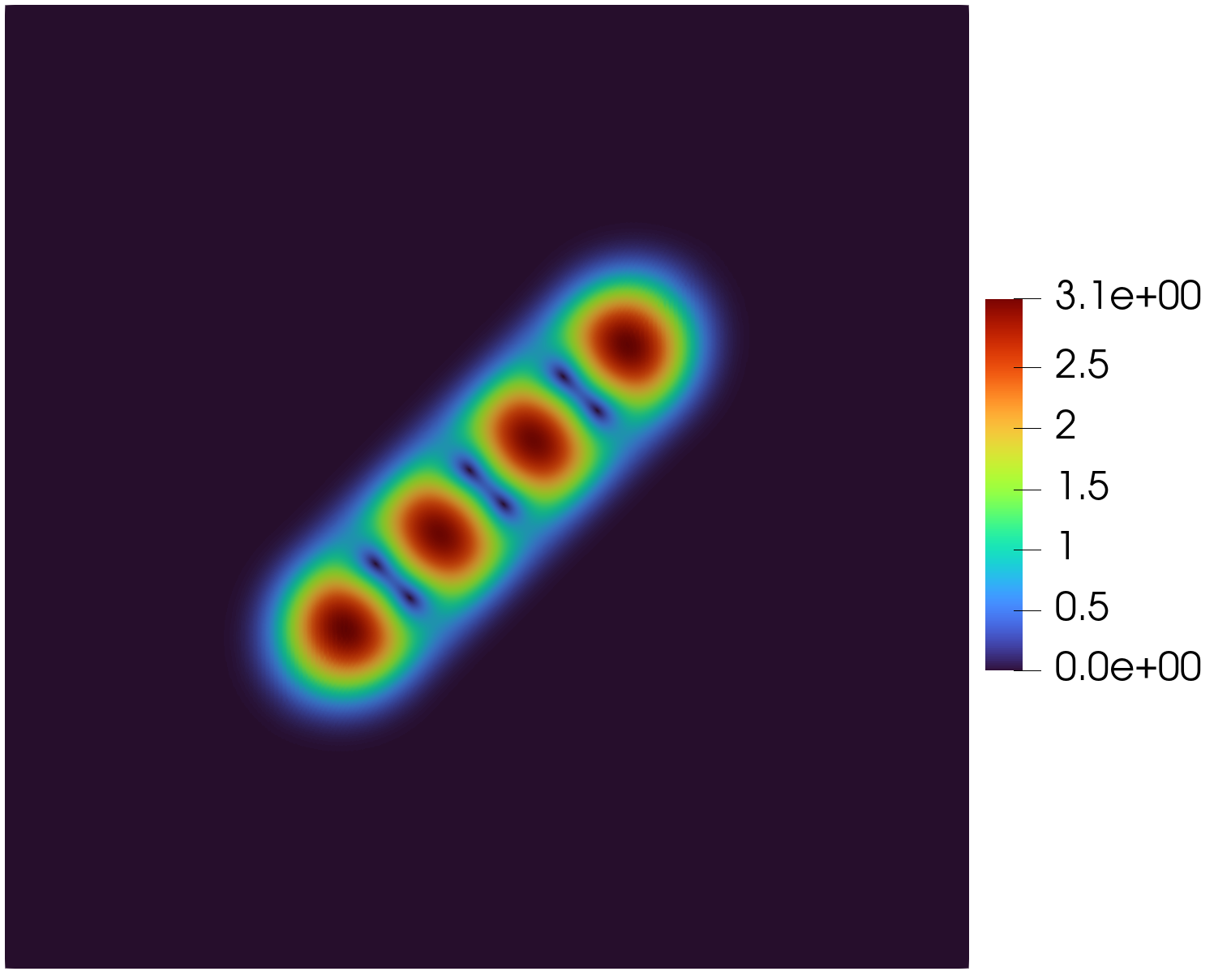}} 
        \caption{\label{SimpleFig} $A$, $|\curl A |$ and $|\psi|$
  for the smallest four eigenvalues. $A$ is given in Example~\ref{Simple}, with $a=1000$, see \S \ref{s:numerics} for details.}
\end{figure}

In general, the identity tells us that, at least for eigenvalues lower in the spectrum, $\nabla \psi(x)$ prefers to point in the direction $-i A(x) \psi(x)$ as much as possible, which does not lead to any obvious a priori restriction on where (or whether) $\psi$ may be localized within the domain $\Omega$. It does, however, suggest that $-i A(x) \psi(x)$ should behave as much as possible like the gradient of a function, at least for smaller eigenvalues. This naturally suggests that subregions where $A$ behaves `almost' like the gradient of a function should be of special interest. We will make this precise.

\subsection{Setup.} The purpose of this section is to introduce and motivate the language in which our main result will be phrased. We will first phrase it informally. This informal formulation will soon be made precise, and will motivate the concepts that we will introduce.\\

\begin{quote}
  \textbf{Main Result} (informal)\textbf{.}
  Let $(\lambda,\psi)$ be an eigenpair of $H(A)$, i.e. $H(A)\psi=\lambda\psi$, and suppose that $|\psi(x_0)| = \| \psi\|_{L^{\infty}}$ for some $x_0\in\Omega$.  Then the vector field $A$ is `close' to conservative in a $\lambda^{-1/2}-$neighborhood of $x_0$.\\
\end{quote}
A casual phrasing would be that localization happens in regions where the vector field $A(x)$ behaves `as much as possible' like the gradient of a function. For readers familiar with the Helmholtz decomposition, we refer to \S 2 for more details. The hallmark of a conservative vector field $F$ is the independence of the path integral. If $F = \nabla \phi$, then a path integral does not depend on the path $\gamma:[0,1] \rightarrow \Omega$, since
$$ \int_{\gamma} F \cdot dx = \phi(\gamma(1)) - \phi(\gamma(0)).$$
Let us now fix two points $x,y \in \Omega$ and a parameter $t > 0$. We will consider Brownian motion $\omega(s): [0,t] \rightarrow \Omega$ conditioned on starting at $x$ and being at $y$ at time $t$, meaning $\omega(0) = x$ and $\omega(t) = y$. Examples of what such random walks could look like are shown in Figure~\ref{fig:walk}.
\begin{center}
\begin{figure}[h!]
\begin{tikzpicture}
\node at (0,0) {\includegraphics[width=0.4\textwidth]{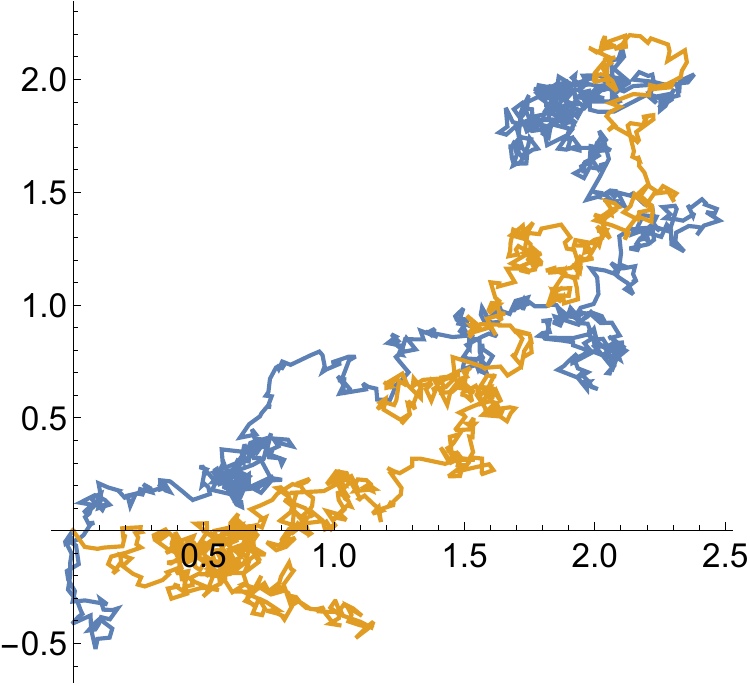}};
\node at (6,0) {\includegraphics[width=0.4\textwidth]{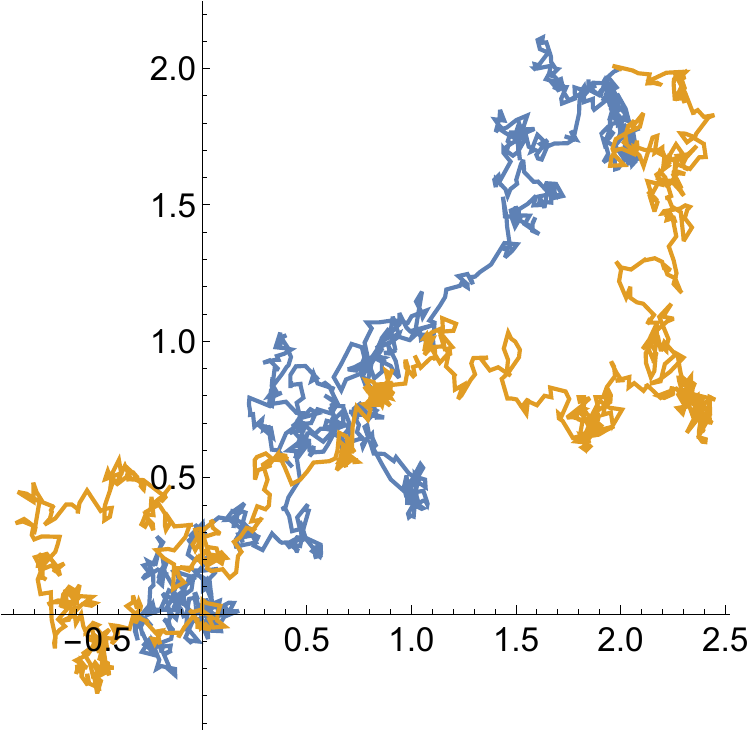}};
\end{tikzpicture}
\caption{Two pictures of two different (equally likely) Brownian motions conditioned to satisfy $\omega(0) = (0,0)$ and $\omega(1) = (2,2)$.}
\label{fig:walk}
\end{figure}
\end{center}

Given a vector field $A(x)$, the two points $x,y \in \Omega$, and Brownian motion conditioned on $\omega(0) = x$ and $\omega(t) = y$, we can now consider, for each such Brownian motion, the corresponding stochastic path integral in the sense of It\^o
$$ \omega \rightarrow \int_{0}^{t} A(\omega(s)) \cdot d\omega(s).$$
This random path integral can be considered as a (real-valued) random variable. If the vector field is irrotational, (meaning $A=\nabla \phi$), and solenoidal (meaning $\nabla\cdot A\equiv 0$), then this path integral is independent of the random walk $\omega$ and is deterministic (meaning that it is a Dirac measure, when considered as a random variable). If $A = \nabla \phi + F$ for some very small vector field $F$, then one might be inclined to expect that the random path integral is presumably close to $\phi(y) - \phi(x)$ but it will probably fluctuate a little around that value, with random fluctuations induced by $F$ and the path. Loosely put, the more irrotational the vector field is, the more this random variable is independent of the path. Conversely, if the vector field has a large solenoidal component then the value of the path integral will presumably depend very strongly on the path taken, and the random variable will be less concentrated around a fixed value. The entire approach is fundamentally nonlocal: considering the Aharonov-Bohm effect \cite{ah}, it stands to reason that some form of non-locality is expected or maybe even required.

\subsection{Main Result.} We can now state our main result. It implies, via Corollary 1, that, if $(\lambda,\psi)$ is an eigenpair of $H(A)$ and $\psi$ assumes its maximum at $x_0$,  then the path integral from $x_0$ to $y$ cannot be `too' random for most points $y$ in a $\lambda^{-1/2}-$neighborhood around $x_0$.  Such path integrals must be somewhat concentrated around some fixed value. There are two ways of achieving this: one such way is by having the vector field be close to a conservative vector field; the other way is by having the integration path be short (which will be equivalent to $\lambda$ being large). We can now formulate our main result for eigenfunctions with Dirichlet boundary conditions. The case of Neumann boundary conditions is, at least in practice, virtually identical (see \S \ref{s:neumann}).

\begin{thm}[Main Result] Let $(\lambda,\psi)$ be an eigenpair of $H(A)$, where $A$ is a vector field with the Helmholtz decomposition
$$ A = \nabla \phi + F \qquad \mbox{where} \quad \div F = 0,$$
subject to the Dirichlet boundary condition $\psi\big|_{\partial \Omega} = 0$. Suppose the eigenfunctions assumes its maximum, $|\psi(x_0)| = \|\psi\|_{L^{\infty}}$, for some $x_0\in\Omega$. Then, for every $t > 0$,
$$\bigintss_{\Omega} \left| \mathbb{E}_{\omega(0) = x_0, \omega(t) = y}  \exp\left( i  \int_{0}^{t} F \cdot d\omega(s) \right) \right|  \frac{1}{(4 \pi t)^{d/2}}\, \exp\!\left(- \frac{\|x_0-y\|^2}{4t}\right) dy \geq e^{-\lambda \, t}.$$
\end{thm}
The Helmholtz decomposition is required for the following reason: given any vector field, one can always add the gradient of a function, $A \rightarrow A + \nabla g$, without changing localization properties of eigenfunctions (see \S 1.7). Thus, any pointwise or localized statement has to respect this type of invariance. This result may look involved at first sight but has a straightforward interpretation, which goes as follows: for $t \sim 1/\lambda$, the right-hand side is close to 1. It is easy to see, via triangle inequality, that the integral is always $\leq 1$, so it must be very close to 1.  Since the Gaussian weight is nonnegative, this forces the absolute value of the expectation to be close to 1 for most $y$. This, in turn, forces the random path integral to be close to a constant value ($\mbox{mod}~2\pi$) for most $y$. This constrains the path integrals in a neighborhood of $x_0$ to be `almost independent of the path' and  `close' to conservative vector fields. In greater detail, the argument proceeds as follows.
\begin{enumerate}
\item If  $t \sim c/\lambda$ for some small $0 < c < 1$, then the right-hand side is $e^{-\lambda t}\sim e^{-c} \sim 1-c$, which is a number very close to $1$. The integral is always $\leq 1$, which, combined with the lower bound, forces the integral to be close to $1$.
\item  The path integral is always real-valued.
Using the triangle inequality, we have
$$ \left| \mathbb{E} \exp(i X) \right| \leq \mathbb{E} \left| \exp(i X) \right| = 1.$$
Note that this inequality is usually strict. Equality is attained if and only if $X$ is constant modulo $2\pi$. Near-equality is attained if and only if $X$ varies little around a fixed value (modulo $2\pi$).
\item Applying the trivial inequality $ \left| \mathbb{E} \exp(i X) \right| \leq 1$, we are left with
  $$e^{-c} = e^{-\lambda \, t} \leq \int_{\Omega} \frac{1}{(4 \pi t)^{d/2}} \,\exp\!\left(- \frac{\|x_0-y\|^2}{4t}\right) dy \leq 1,$$
  for $t=c/\lambda$.
\item This, in turn, means that the trivial inequality $ \left| \mathbb{E} \exp(i X) \right| \leq 1$ has to be close to sharp for most $y$, which can only happen if the random variable is close to a constant (modulo $2\pi$) for most $y$.
\end{enumerate}

This argument can be made precise in a variety of ways. One way of making it precise (though, certainly not the only one) is by showing that localization near $x_0$ implies that the random path integral is tightly concentrated when going from $x_0$ to $y$ for most points $y$ close to $x_0$. We note again the two main ways a random path integral can be highly concentrated:
\begin{enumerate}
\item either the vector field $A$ is close to conservative (meaning $F$ is small), or
\item the integration path is relatively short (in which case the value is going to concentrate around 0), which happens when $t$ is small.
\end{enumerate}
Naturally, any combination of the two factors can also occur. We also note that the second factor shows that the presence of a very large solenoidal vector field implies that the eigenvalues $\lambda$ cannot be very small, because $t \sim \lambda^{-1}$ cannot be very large. 

\begin{corollary} \label{cor:1}
  There exists a universal constant $0 < c < 1$, depending only on the spatial dimension $d$, for which the following holds.
Let $(\lambda,\psi)$ be an eigenpair of $H(A)$, subject to the Dirichlet boundary condition $\psi\big|_{\partial \Omega} = 0$, with Helmholtz decomposition $A = \nabla \phi + F$ where $\div F =0$. Suppose that $|\psi(x_0)| = \|\psi\|_{L^{\infty}}$ for some $x_0\in\Omega$.  Let $t=c/\lambda$.  We say that $y \in \Omega$ is near-deterministic if any path integral beginning at $x_0$ and conditioned on $\omega(t) = y$ is likely to end up close to some fixed value (mod $2\pi$), in the following sense: 
$$ \sup_{z \in \mathbb{T}} \mathbb{P}\left(  \left|  \int_{0}^{t} F \cdot d\omega(s) \mod 2\pi  - z \right| \leq \frac{1}{100} \right) \geq \frac{99}{100}.$$
It holds that a large fraction of points $y$ in a $\sqrt{t}-$neighborhood of $x_0$ are near-deterministic,
$$ \left| \left\{y \in B_{\sqrt{t}}(x_0): y~\emph{near-deterministic} \right\} \right| \geq  \frac{9}{10}\cdot  \left|  B_{\sqrt{t}}(x_0) \right|.$$
\end{corollary}
We emphasize again that this is only one of many possible ways of deducing a rigorous concentration result from the Main Theorem; many others are possible. We also note that the Main Result can be obtained by using the heat kernel $p_t(x,y)$ of the bounded domain $\Omega$, instead of the Gaussian weight (the heat kernel on $\mathbb{R}^d$). This slightly stronger result might be advantageous in certain settings.

\subsection{Using $\curl$ as a proxy.} For ease of exposition, we assume now that $\Omega \subset \mathbb{R}^2$ throughout this subsection.
 Given a vector field $A$, we perform again a Helmholtz decomposition and write it as
\begin{align*}
 A(x) = \nabla \phi + F(x) \qquad \mbox{where} \quad \phi \in C^1(\Omega) \quad ~\mbox{and} ~\div(F) = 0.
\end{align*} 
The vector field $\nabla \phi$ is conservative, and has no impact on localization properties (see \S \ref{s:helm} for a proof). This can also be seen as a gauge invariance, as described in \S \ref{s:helm}, which shows  that the only effect of the quantity $\nabla \phi$ is to modulate the eigenfunction, $\psi \rightarrow e^{-i \phi} \psi$. As discussed in the main Theorem and Corollary 1, one would expect localization to occur in regions where $F(x)$ behaves `almost' like a conservative vector field. These regions should be the ones where $\left|\curl(F)\right|$ is relatively small. Numerical examples (\S~\ref{s:numerics}) show that this is a reasonable heuristic in practice.  We recall that the curl of a 2D vector field $F=(F_1,F_2)$ is the scalar field $\curl(F)=\partial F_2/\partial x_1-\partial F_1/\partial x_2$.

\begin{figure}
  \centering
	\subfloat[{$A$.}]
	{\includegraphics[width=0.35\textwidth]{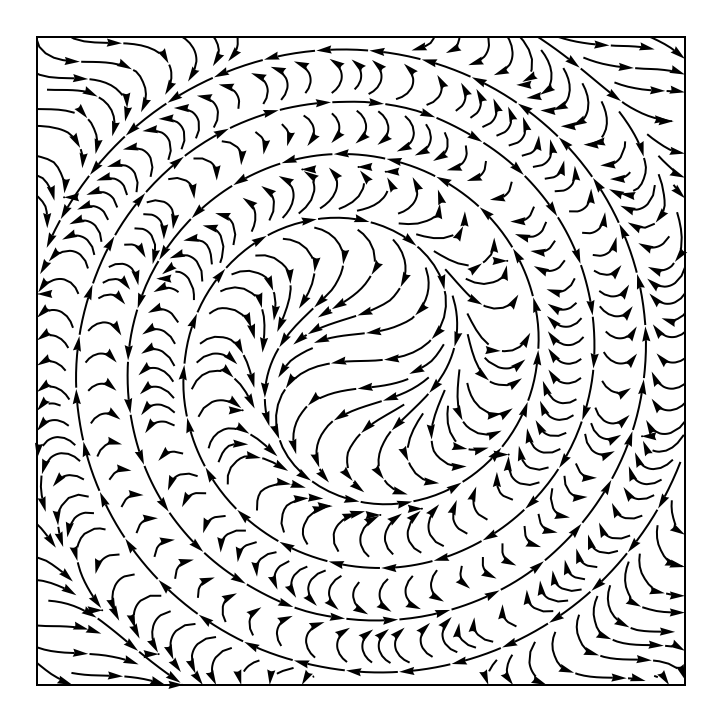}}
        \quad
 	\subfloat[{$|\curl A|$.}]
	{\includegraphics[width=0.39\textwidth]{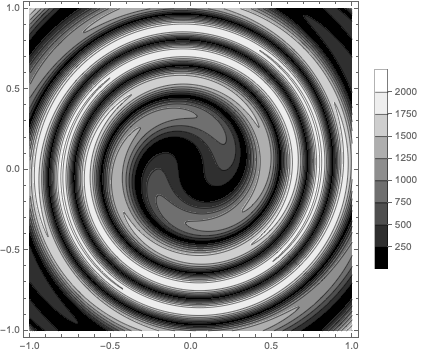}}
        \\
 	\subfloat[{ $\lambda\approx 92.8$.}]
	{\includegraphics[width=0.30\textwidth]{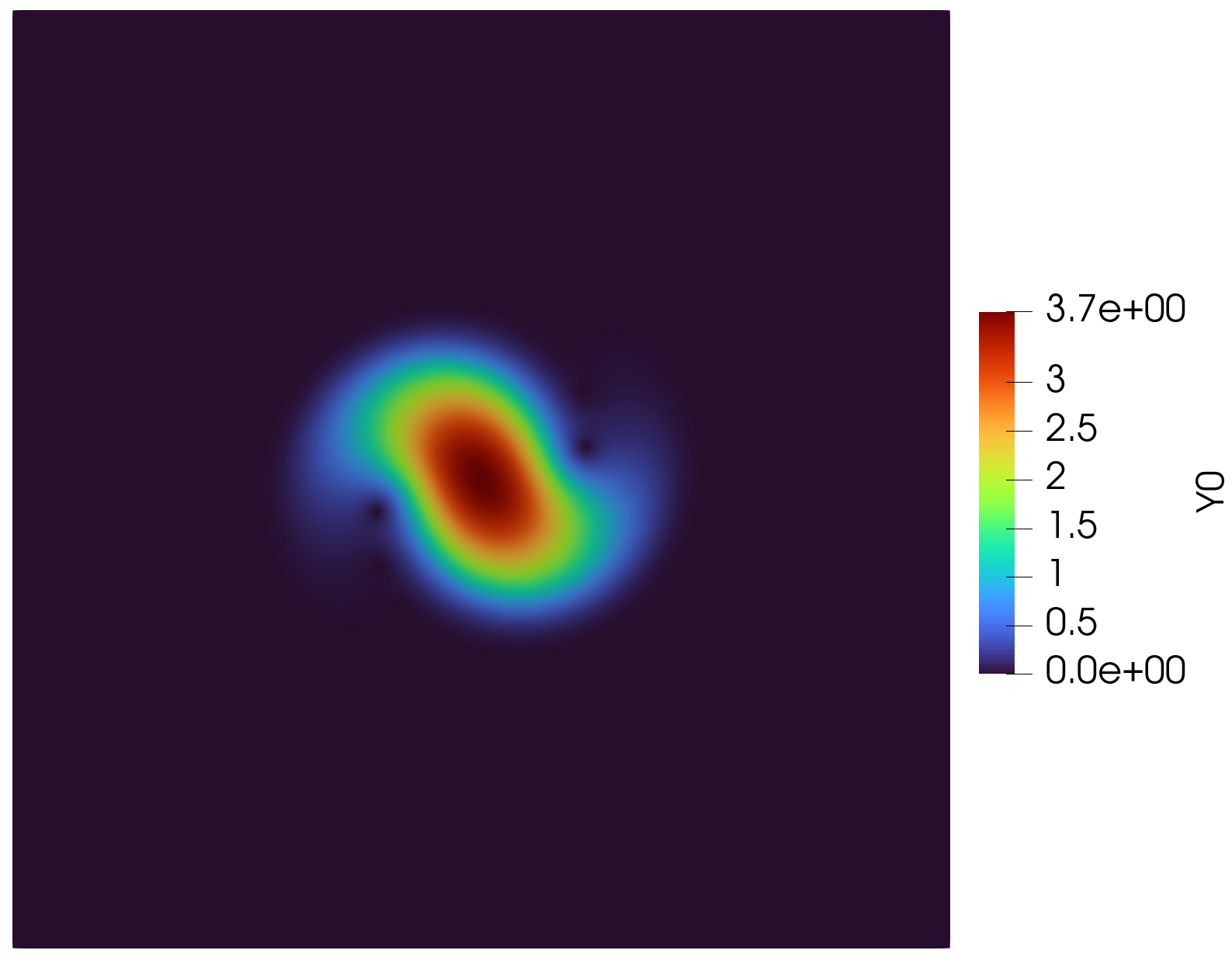}}
        \quad
 	\subfloat[{ $\lambda\approx 132.6$.}]
	{\includegraphics[width=0.30\textwidth]{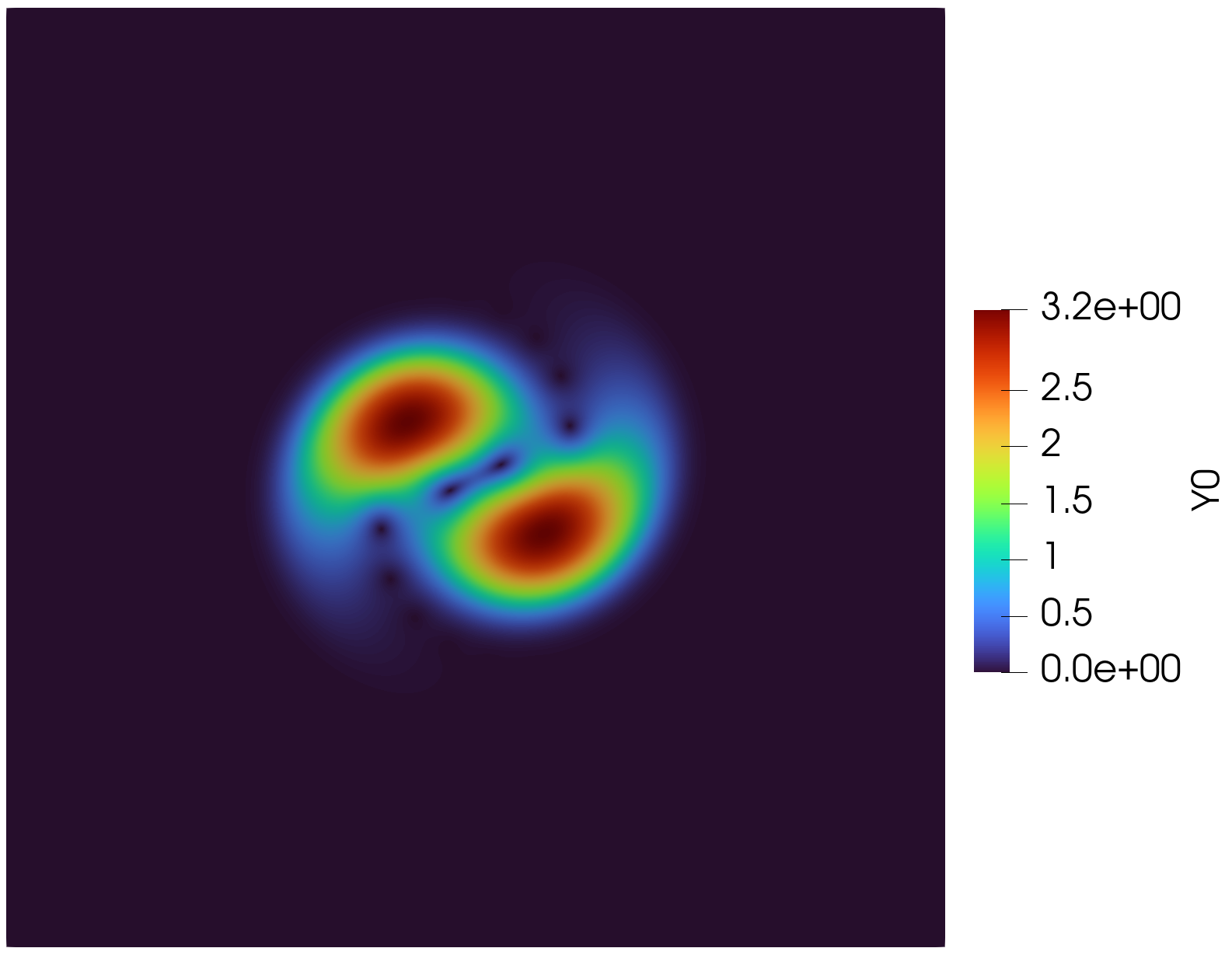}}
        \quad
        \subfloat[{ $\lambda\approx 163.8$.}]
	{\includegraphics[width=0.30\textwidth]{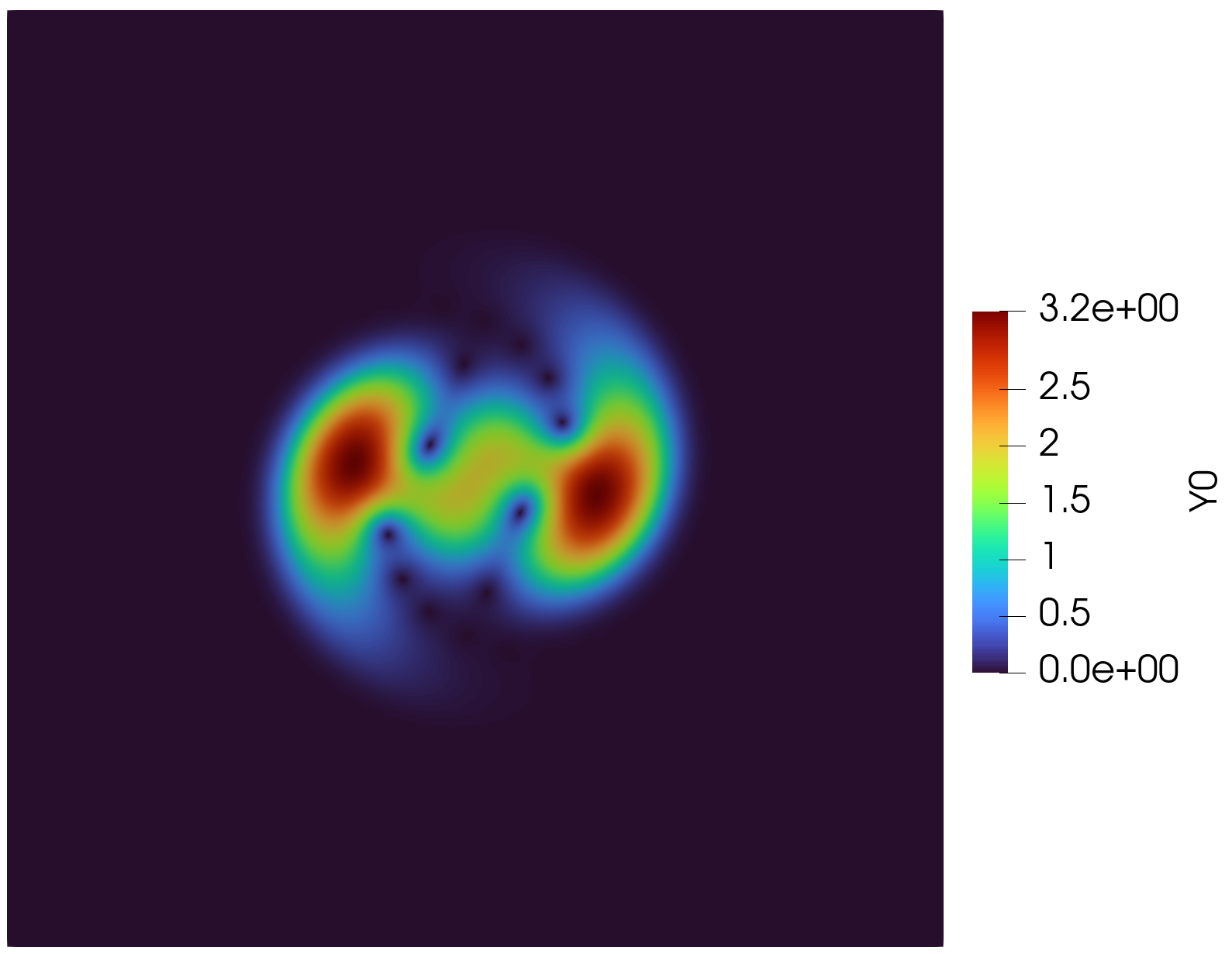}}
        \\
 	\subfloat[{ $\lambda\approx 167.2$.}]
	{\includegraphics[width=0.30\textwidth]{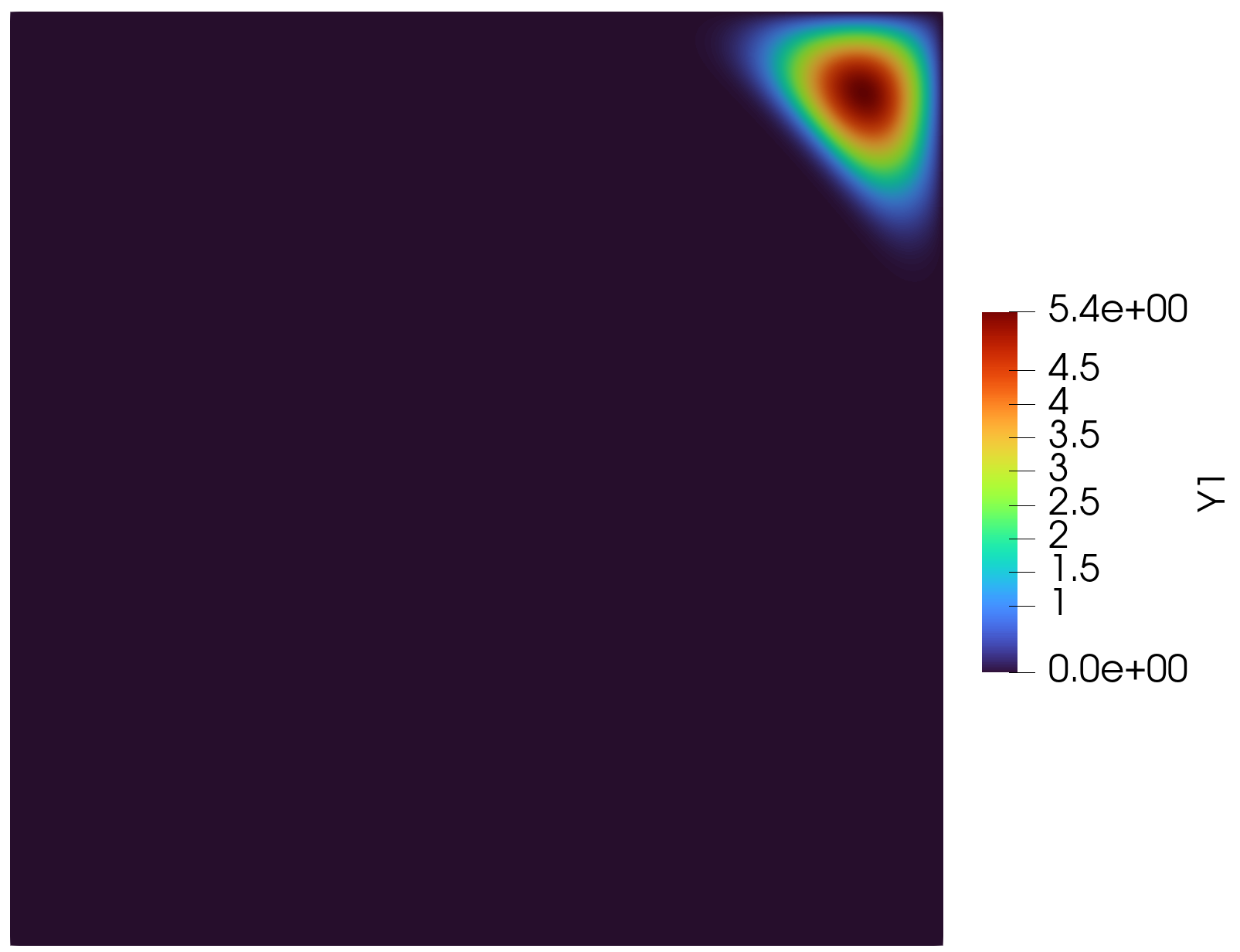}}
        \quad
        \subfloat[{ $\lambda\approx 167.2$.}]
	{\includegraphics[width=0.30\textwidth]{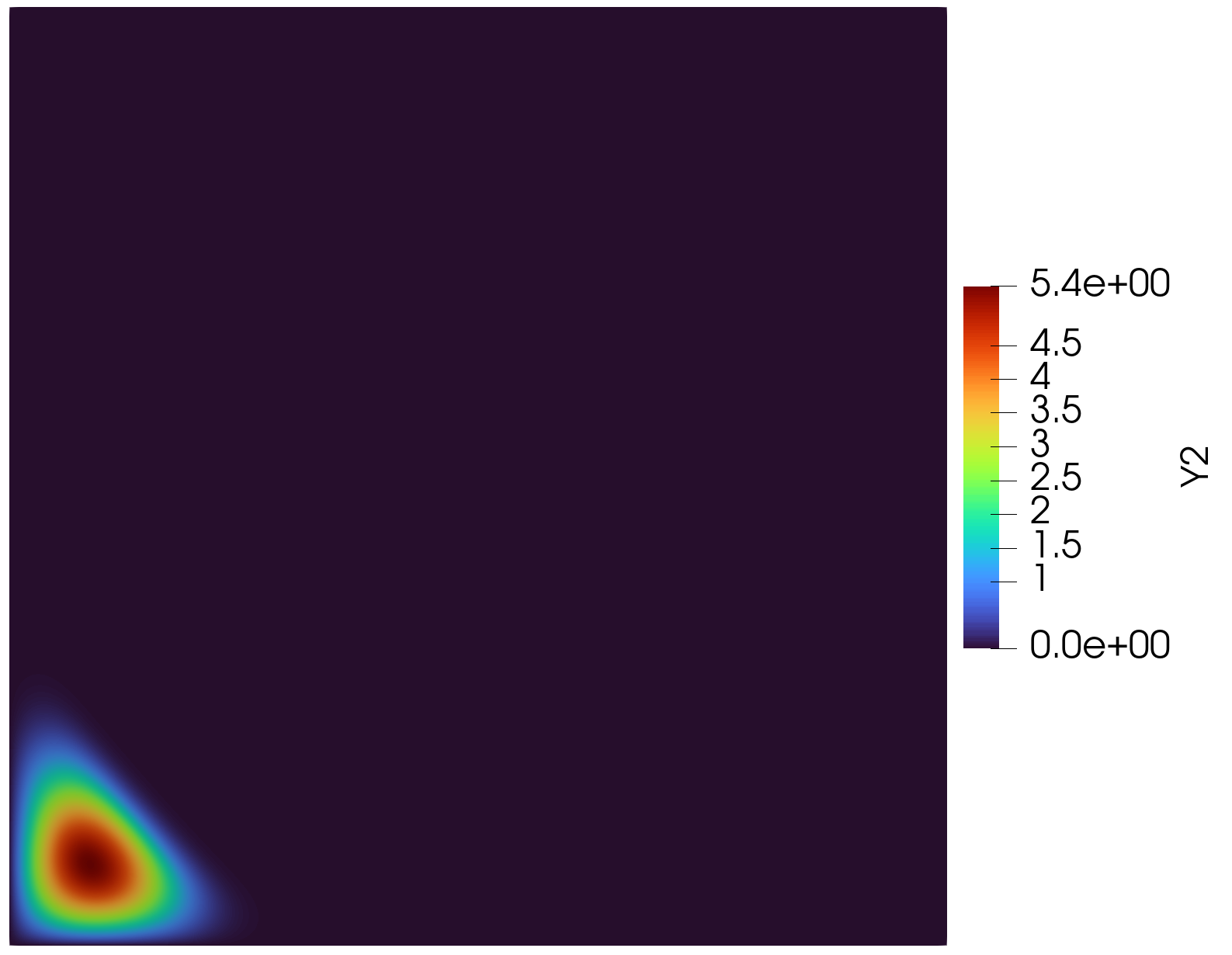}}
       \quad
 	\subfloat[{ $\lambda\approx 185.5$.}]
	{\includegraphics[width=0.30\textwidth]{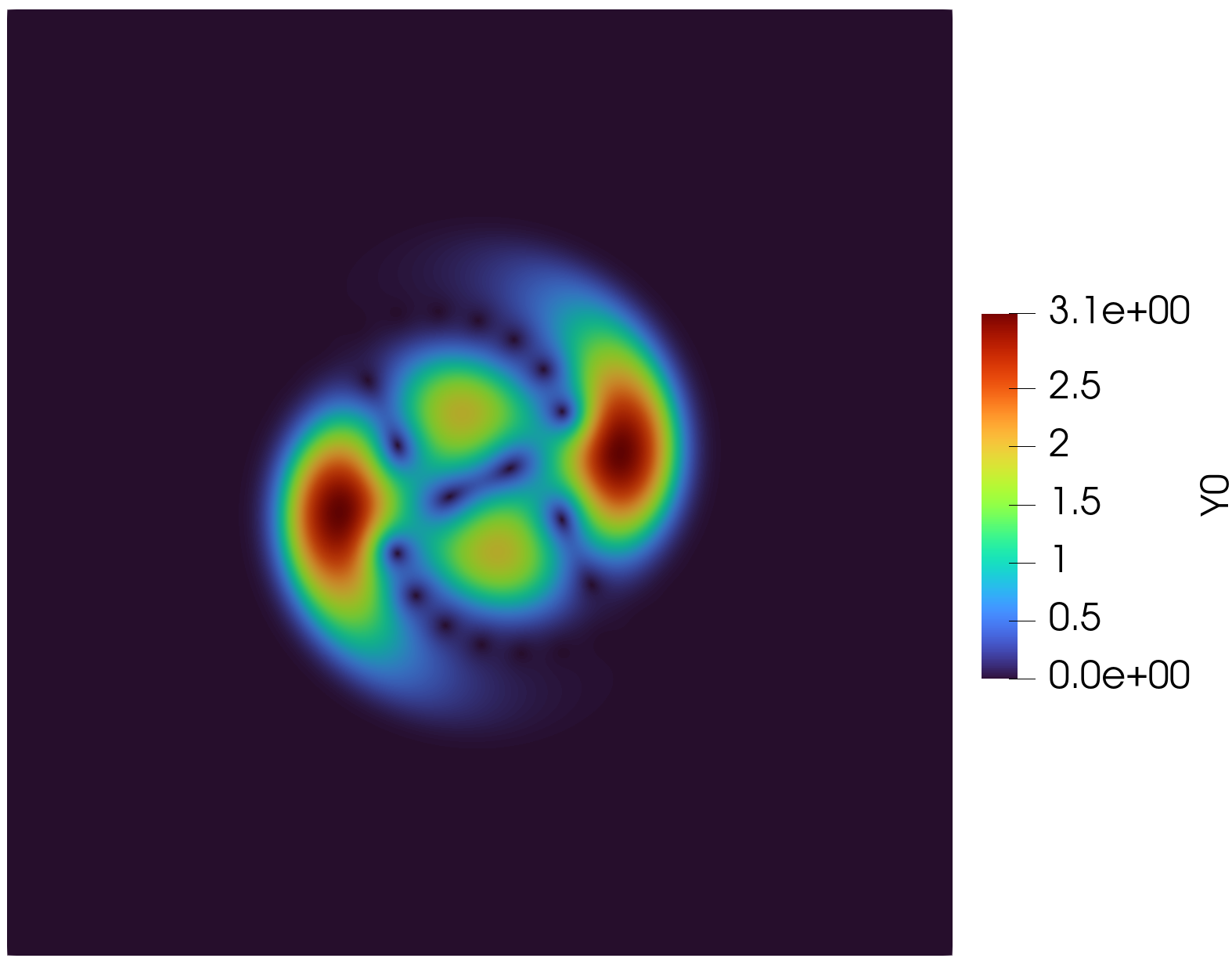}}
        \\
 	\subfloat[{$\lambda\approx 205.3$.}]
	{\includegraphics[width=0.30\textwidth]{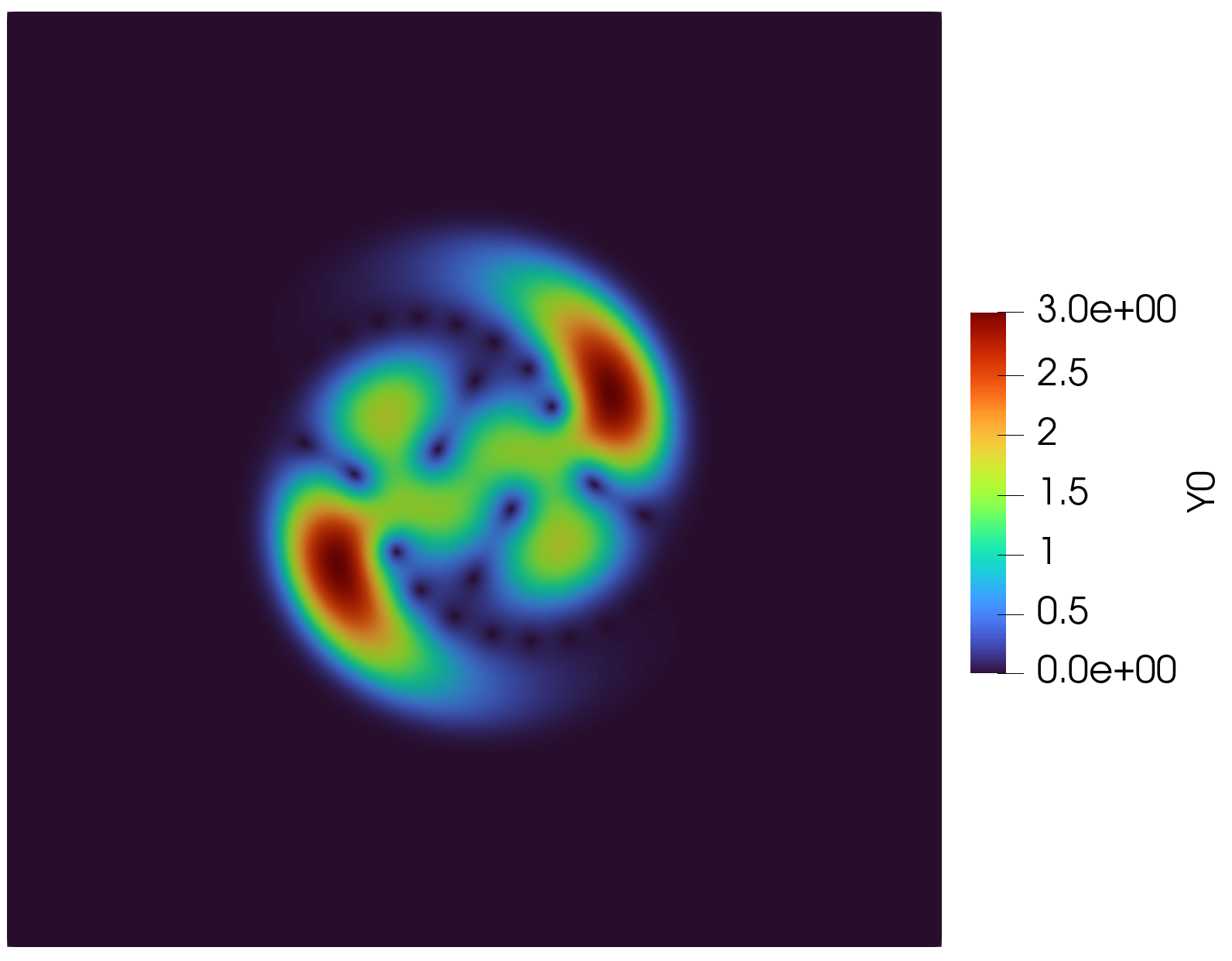}}
        \quad
        \subfloat[{ $\lambda\approx 220.5$.}]
	{\includegraphics[width=0.30\textwidth]{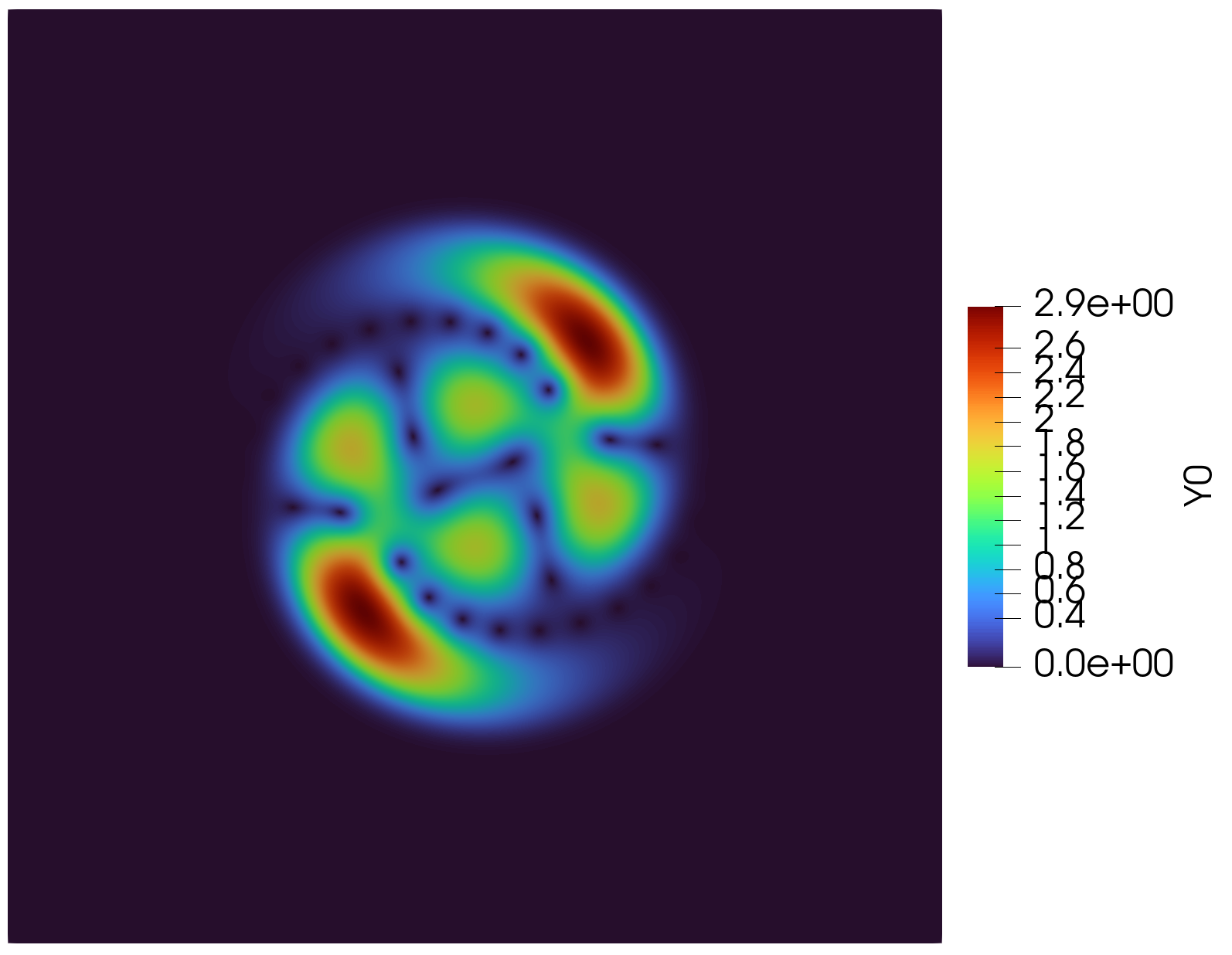}}
       \quad
 	\subfloat[{$\lambda\approx 234.8$.}]
	{\includegraphics[width=0.30\textwidth]{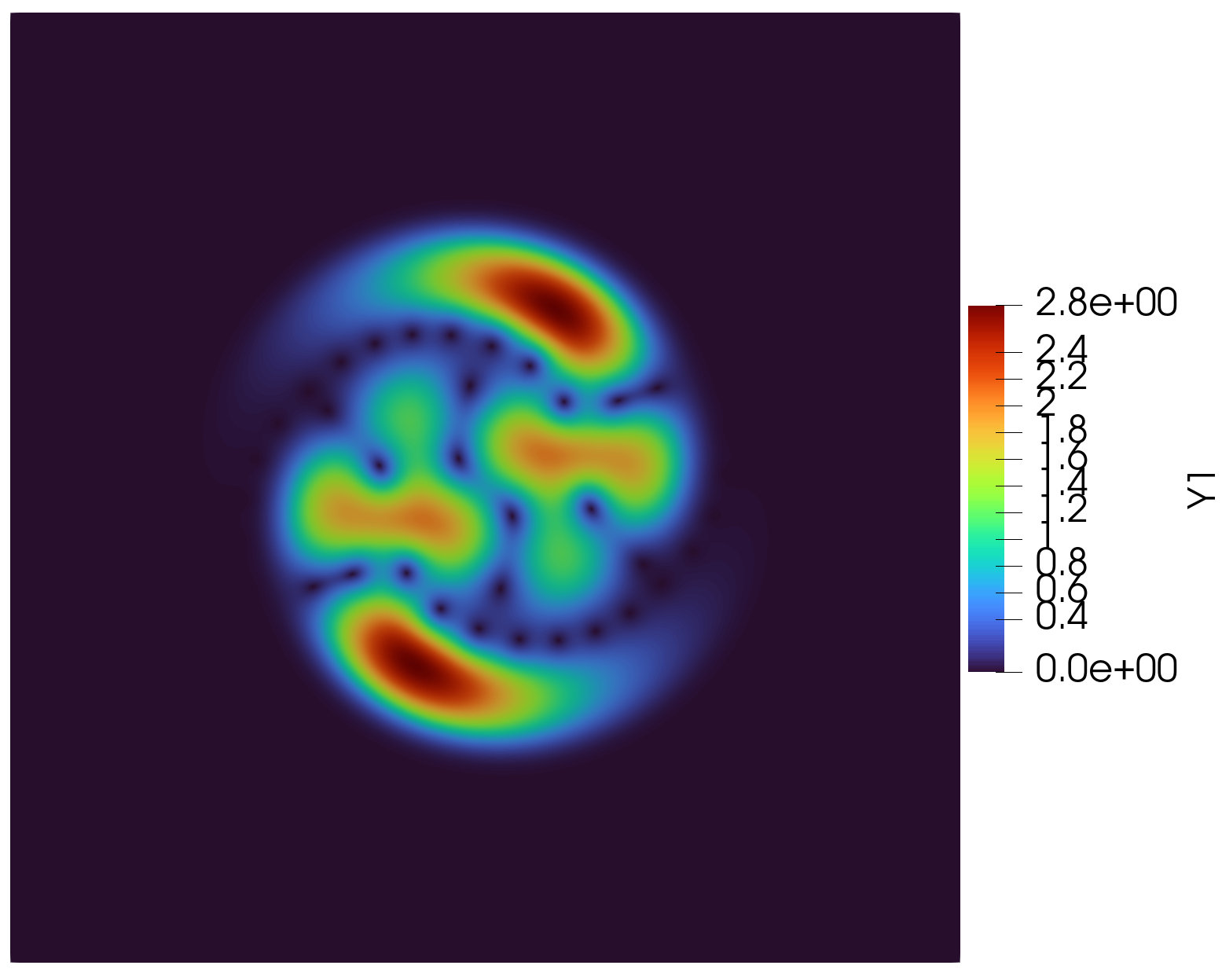}}  
  \caption{\label{Example3Fig} $A$, $|\curl A|$ and $|\psi|$
  for the smallest nine eigenvalues. $A$ is given in Example~\ref{Example3}, with $a=50$, see \S \ref{s:numerics}.}
\end{figure}

\vspace*{2mm}
Given $x_0\in\Omega$, let $A_{\mbox{\tiny lin}}(x)=A(x_0)+J(x_0)(x-x_0)$ be the linearization of $A$ about $x_0$.  Here, $J$ is the Jacobian matrix of $A$.  We then define $A_{\mbox{\tiny nonlin}} =A-A_{\mbox{\tiny lin}}$.
In anticipation of the proof of Corollary~\ref{cor:2}, and to give an early indication of the importance of the curl, we note that 
$A_{\mbox{\tiny lin}}=A_1+A_2$, where
\begin{equation}\label{LinearSplitting}
\begin{split}
  A_1(x)&=A(x_0)+\frac{1}{2}\left(J(x_0)+J(x_0)^T\right)(x-x_0)
\;,\;
R=\begin{pmatrix}0&-1\\1&0\end{pmatrix}~,\\
  A_2(x)&=\frac{1}{2}\left(J(x_0)-J(x_0)^T\right)(x-x_0)=\frac{\curl A(x_0)}{2}\,R(x-x_0)~.
\end{split}
\end{equation}
It is clear that $A_1$ is conservative and $A_2$ is solenoidal.

\begin{corollary} \label{cor:2}
  Let $(\lambda,\psi)$ be an eigenpair of $H(A)$, subject to the Dirichlet boundary condition $\psi\big|_{\partial \Omega} = 0$, and suppose that $|\psi(x_0)| = \|\psi\|_{L^{\infty}}$ for some $x_0\in\Omega$.  If $A_{\mbox{\tiny nonlin}} $
 is sufficiently small in a $1/\sqrt{\lambda}-$neighborhood of $x_0$
then
$$  \left| \curl A(x_0) \right| \leq c \, \lambda^2. $$
\end{corollary}
This result is proven via an asymptotic expansion of the random path integral up to linear terms (hence also the restriction that the linear terms of the vector field dominate). It is not a priori clear whether this type of decomposition is optimal, and the problem of how to best utilize the main result for numerical prediction remains an interesting problem. Nonetheless, as illustrated by various examples throughout the paper, using sublevel sets of  $|\mathrm{curl}(A)|$ leads to very reasonable predictions of where eigenvectors early in the spectrum are likely to localize.

\subsection{A landscape inequality.} While all these arguments are presumably most powerful in the regime where $t$ is chosen to be small, $t \sim 1/\lambda$, there are averaging arguments one could employ. One such averaging argument leads to a refinement of the landscape inequality that again illustrates our main point. Applying the landscape inequality \cite{hoskins}, we deduce that any
eigenpair $(\lambda,\psi)$ of $H(A)$ satisfies
$$ |\psi(x)| \leq \lambda  \|\psi\|_{L^{\infty}} \, v(x) ~,$$
where $v$ is the classical torsion function, $-\Delta v = 1$.  Both $\psi$ and $v$ are assumed to satisfy Dirichlet boundary conditions.
This inequality does not involve the vector field $A$ (it only depends on the domain $\Omega$), so one would perhaps not expect it to be very informative in general.  A  notable exception is when the vector field $A$ is irrotational, in which case the inequality should be very accurate in predicting localization of the ground state, as seen in Lemma~\ref{lem:nice} below.
We derive a small refinement of this inequality. For the sake of simplicity of exposition, we shall abbreviate the stochastic path integral as
$$ \mathbb{E}_{x,y}(t) = \mathbb{E}_{\omega(0) = x_0, \omega(t) = y}  \exp\left( i  \int_{0}^{t} F \cdot d\omega(s) \right) \in \left\{z \in \mathbb{C}: |z| = 1\right\}. $$
The main result, by integrating over time, then leads to an improved landscape inequality, which is always at least as good as the original result.
\begin{corollary} \label{cor:3}
We have
$$ |\psi(x)| \leq \lambda  \|\psi\|_{L^{\infty}} \cdot \left[ \left(  \int_0^{\infty}  \int_{\Omega} |\mathbb{E}_{x,y}(t)|^2 \cdot p_t(x,y) dy \,dt\right)^{1/2}  \sqrt{v(x)} \right],$$
where the first integral can be bounded from above by
$$ \left(  \int_0^{\infty} \int_{\Omega}  |\mathbb{E}_{x,y}(t)|^2  p_t(x,y) dy \,dt\right)^{1/2} \leq \left(  \int_0^{\infty}  \int_{\Omega} p_t(x,y) dy \,dt\right)^{1/2} \leq \sqrt{v(x)}.$$
\end{corollary}

This inequality is more complicated, as it involves the heat kernel $p_t(x,y)$. One could, of course, again bound it from above by the Gaussian heat kernel in free space. There is no reason to believe that Corollary \ref{cor:3} is particularly useful for general $A = \nabla \phi + F$.
However, it is yet another way to illustrate our main point: one can improve on the landscape inequality in regions where $|\mathbb{E}_{x,y}(t)| \ll 1$.  These are the regions where the vector field is far from path-independent.

\subsection{The Helmholtz decomposition.} \label{s:helm} The idea that $\nabla \psi(x) \sim -i A(x) \psi(x)$ should be true over most of the domain does suggest a natural idea: some vector fields arise naturally as the gradient of a function (the conservative vector fields). This suggests performing a Helmholtz decomposition
\begin{align*}
 A(x) = \nabla \phi + F(x) \qquad \mbox{where} \quad \phi \in C^1(\Omega) \quad ~\mbox{and} ~\div(F) = 0.
\end{align*} 
As it turns out, the Helmholtz decomposition leads to a natural symmetry of the problem: the irrotational contribution $\nabla \phi$ to the vector field $A$ does not impact localization properties of the eigenfunction, it only leads to a modulation (multiplication by complex numbers with modulus 1). 

\begin{lemma} \label{lem:nice} Given the Helmholtz decomposition $A(x) = \nabla \phi + F(x)$, define the operator $\tilde H g = e^{-i \phi} H (e^{i \phi} g)$. It holds that $\tilde H = (-i \nabla -F)^2$.
Therefore, $H\psi = \lambda \psi$ if and only if $\tilde H(e^{-i \phi} \psi) = \lambda e^{-i \phi} \psi$.
\end{lemma}

This lemma, whose proof follows by direct computation, has a number of implications. The most immediate is perhaps that, if one cares about localization, the contribution $\nabla \phi$ to the vector field is completely irrelevant.  One can always perform a Helmholtz decomposition and simply remove the contribution that is the gradient of a function.  This procedure has no impact on localization behavior of eigenfunctions, though it does have an impact on their modulation via $\psi \rightarrow e^{-i \phi} \psi$.
This allows us to rephrase the question from above.

 \begin{quote}
   \textbf{Problem} (simplified)\textbf{.}
   Given $F$ satisfying $\mbox{div}(F) = 0$, predict whether and where eigenfunctions of $H(F)$ may localize, based on $F$ and eigenvalues $\lambda$.
 \end{quote}

\subsection{Neumann boundary conditions} \label{s:neumann} The main result is formulated for eigenfunctions satisfying Dirichlet boundary conditions. The case of Neumann boundary conditions is virtually identical and, in practice, there is little difference between the two \textit{provided} the eigenfunction in question is localized inside the domain (and at least several wavelengths away from the boundary). This can also be seen from the proof: the proof uses the Feynman-Kac formula to yield a local reproducing identity of a parabolic nature. Such an argument is, by default, not localized away from the boundary but has the usual exponential decay at the scale of a wavelength. Thus, whatever actually happens at the boundary is of little concern. This, of course, changes dramatically when the eigenfunction is localized near the boundary, in which case a nontrivial type of interaction can occur (see, for example,  \cite{felix, jones}).

\begin{figure}[h!]
         \subfloat[{$\Re\psi$, $\lambda\approx 120.52$.}]
	{\includegraphics[width=0.48\textwidth]{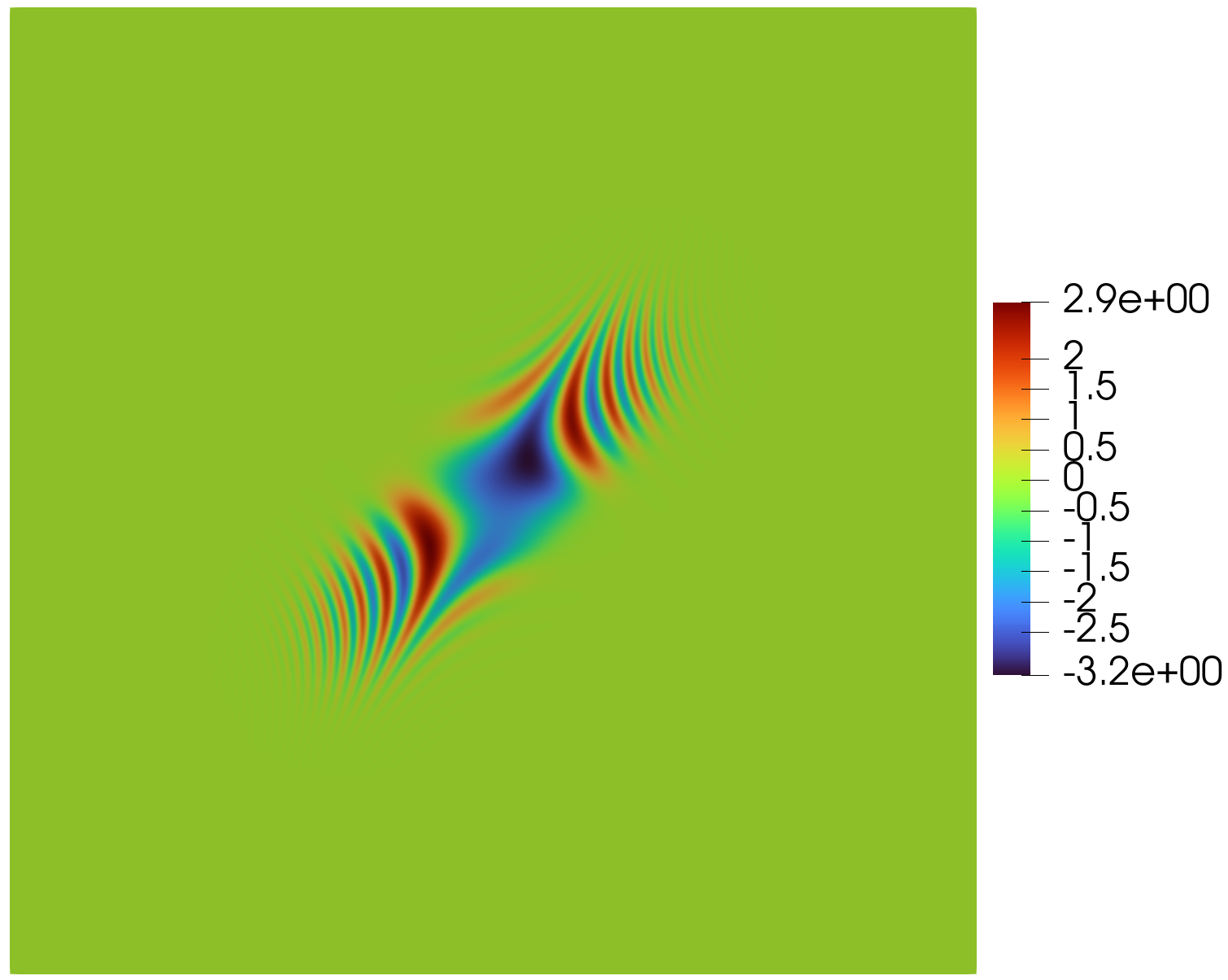}}
        \quad
 	\subfloat[{$\Im\psi$, $\lambda\approx 120.52$.}]
	{\includegraphics[width=0.48\textwidth]{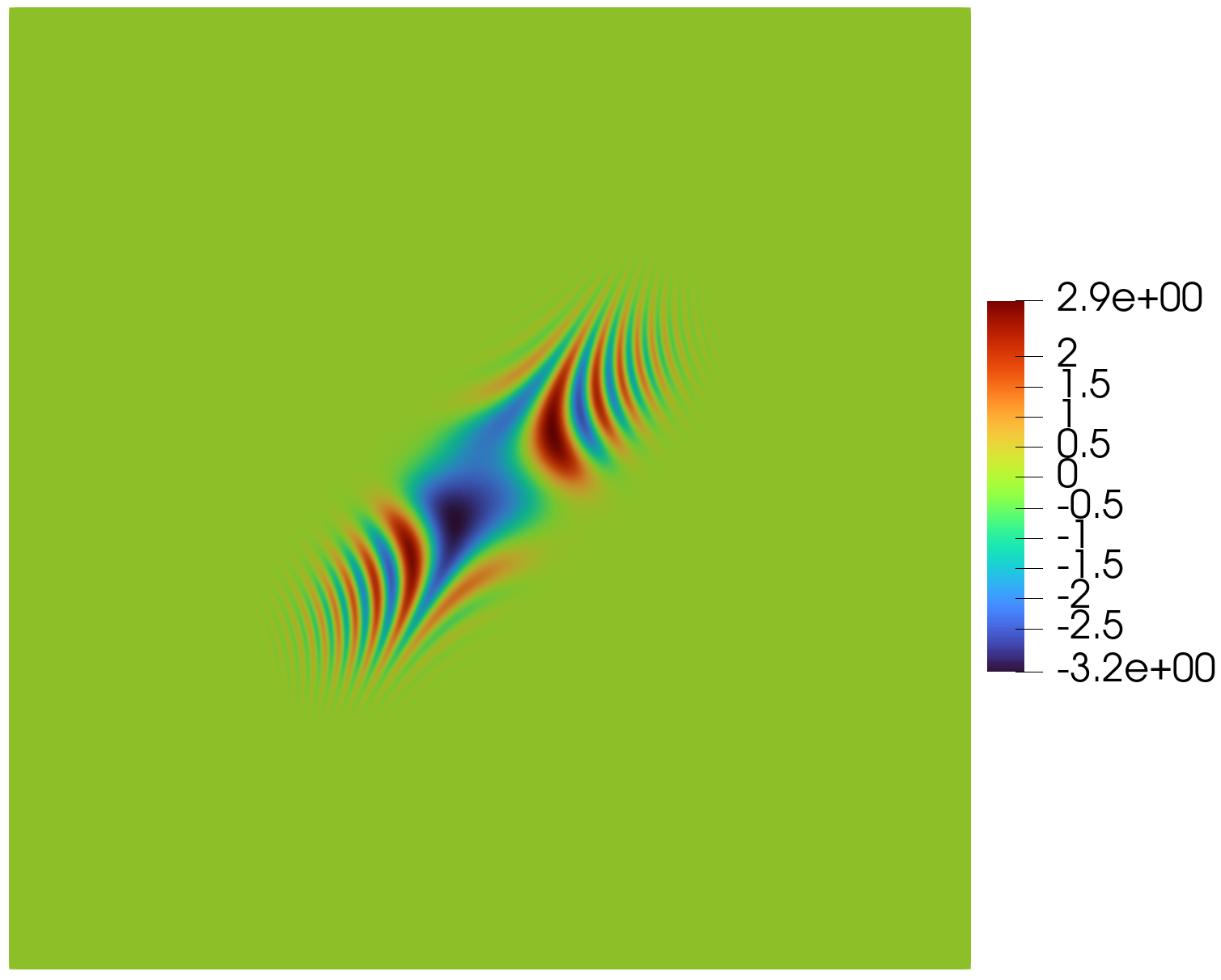}} 
        \caption{\label{SimpleFigB} $\Re\psi$ and $\Im\psi$ for the
          smallest eigenvalue. $A$ given by Example~\ref{Simple}, with $a=1000$, see \S \ref{s:numerics}.}
\end{figure}

\subsection{Gauge freedom and parallel transport}

That several of these results have conclusions which are stated modulo $2\pi \mathbb{Z}$ should also not be surprising, and are in fact suggested by the gauge freedom of the operator which Lemma \ref{lem:nice} implies. The origin of this gauge freedom lies in the fact that our magnetic Laplacian
\[ H(A) = \left(- i\nabla - A(x)\right)^2 \]
arises as a \emph{connection Laplacian} acting on sections of a complex line bundle $E$. A complex line bundle over a simply-connected domain $\Omega$ has a trivialization $E \simeq \Omega\times \mathbb{C}$, thus there is a non-vanishing smooth section $s_0: \Omega \to E$.  More specifically, given a function $f:\Omega \to \mathbb{C}$, the product $s_0f$ now defines a section of $E$, and every section arises this way. Similarly, the magnetic potential $iA$ (now viewed as a complex-valued 1-form on $\Omega$), defines a connection on $E$, $\nabla^A$, by the formula $\nabla^A(s_0 f)=s_0\,(d+iA)f$, where $d$ is the exterior derivative. The connection Laplacian defined by $\nabla^A$ is precisely our 
magnetic Laplacian operator,
\[ \Delta_A := (\nabla^A)^*\nabla^A = \left(- i\nabla - A(x)\right)^2 \]
acting on sections of $E$ (equivalently on complex-valued functions).
This formalism can seem cumbersome, but it does clarify the relationship between the integral $\int_0^t A(\omega(t))\cdot d\omega(s)$ and $H(A)$. For each $p\in \Omega$ and closed loop $\gamma$ based at $p$, a choice of connection $\nabla^A$ defines an endomorphism $\mathcal{P}_\gamma^A: \mathbb{C}\to \mathbb{C}$ of the fiber of $E$ at $p$, known as the \emph{parallel transport map}. For the connection $\nabla^A=d+iA$, an easy computation shows that
\[ \mathcal{P}_\gamma^A = e^{i\oint_\gamma A}  \]
for the loop $\gamma$. We note that any connection is also determined by its parallel transport map evaluated on all smooth loops based at a point \cite{kobayashi}.
In particular, this relates parallel transport to the gauge freedom Lemma \ref{lem:nice} illustrates. Namely our connections are gauge-equivalent, $e^{-i\phi}(d+iA_1)e^{i\phi}=(d+iA_2)$, if and only if $A_1-A_2=d\phi$. Now because $A_1-A_2$ and $d\phi$ define the same connection, this is equivalent to these connections defining the same parallel transport maps, thus
\[ \mathcal{P}_\gamma^{A_1-A_2} = \mathcal{P}_\gamma^{d\phi} = e^{i\oint_\gamma d\phi} = 1,  \]
for all closed loops $\gamma$, by Stokes theorem. In particular we see that that $A_1, A_2$ define gauge equivalent connections in this setting if and only if
\[ \oint_\gamma A_1-A_2 \in 2\pi \mathbb{Z} . \]
Note that we have implicitly used the fact that our domain $\Omega$ is simply connected; for non-simply connected domains the connection may not take the form $d+iA$ globally and the corresponding parallel transport map will change in turn. 


\begin{figure}[h!]
  \centering
	\subfloat[{$A$.}]
	{\includegraphics[width=0.35\textwidth]{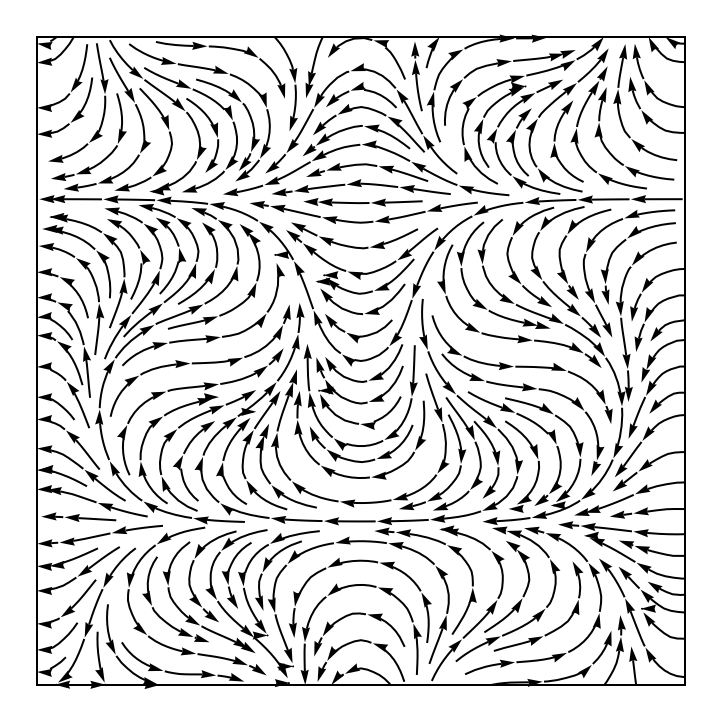}}
        \quad
 	\subfloat[{$|\curl A|$.}]
	{\includegraphics[width=0.39\textwidth]{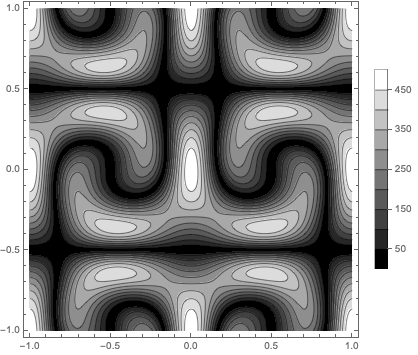}}\\
        \subfloat[{ $\lambda\approx 62.030$.}]
	{\includegraphics[width=0.30\textwidth]{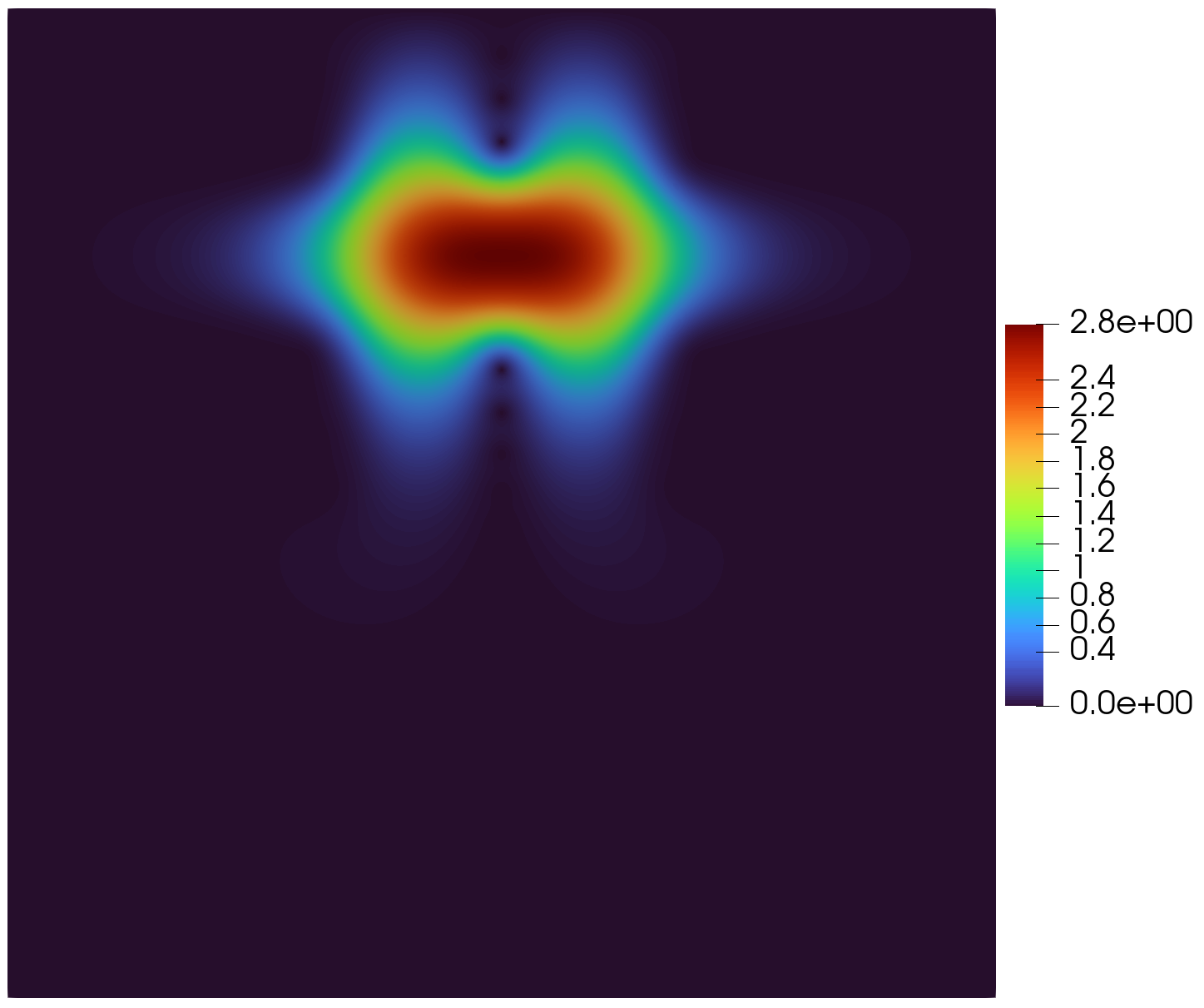}}
         \quad
         \subfloat[{$\lambda\approx 85.610$.}]
 	{\includegraphics[width=0.30\textwidth]{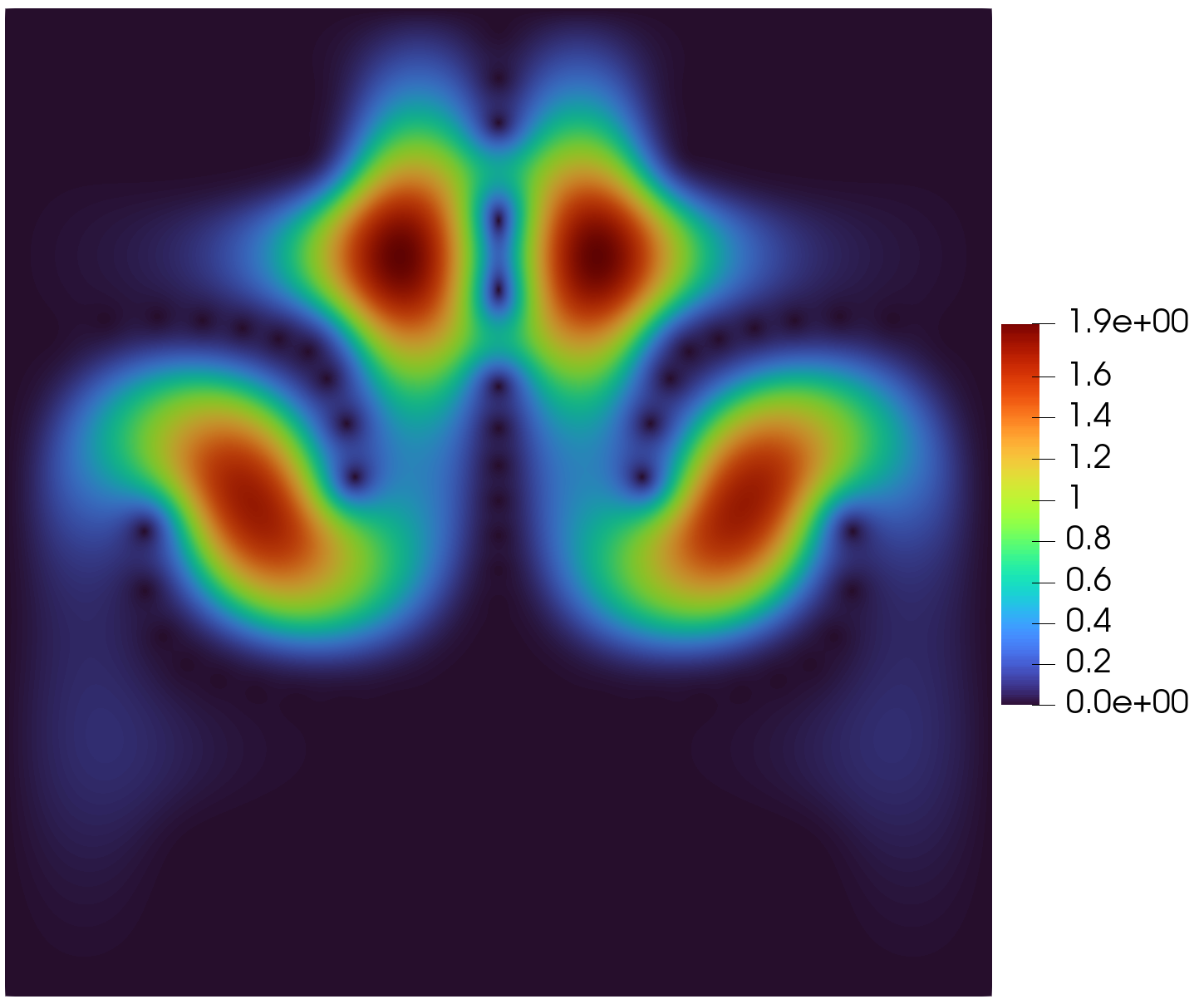}}
        \quad
  	\subfloat[{$\lambda\approx 85.875$.}]
 	{\includegraphics[width=0.30\textwidth]{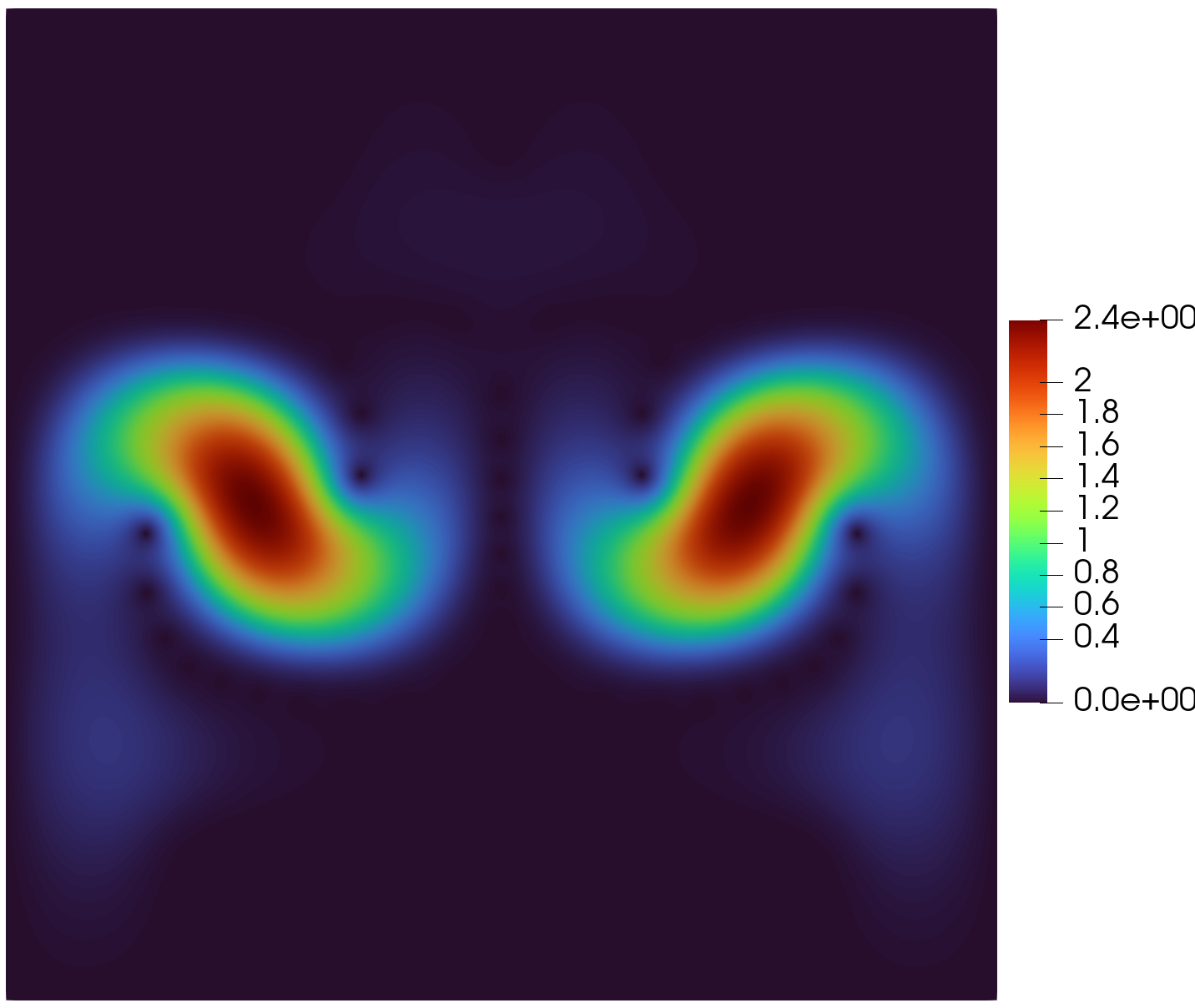}}\\
        \subfloat[{$\lambda\approx 86.150$.}]
	{\includegraphics[width=0.30\textwidth]{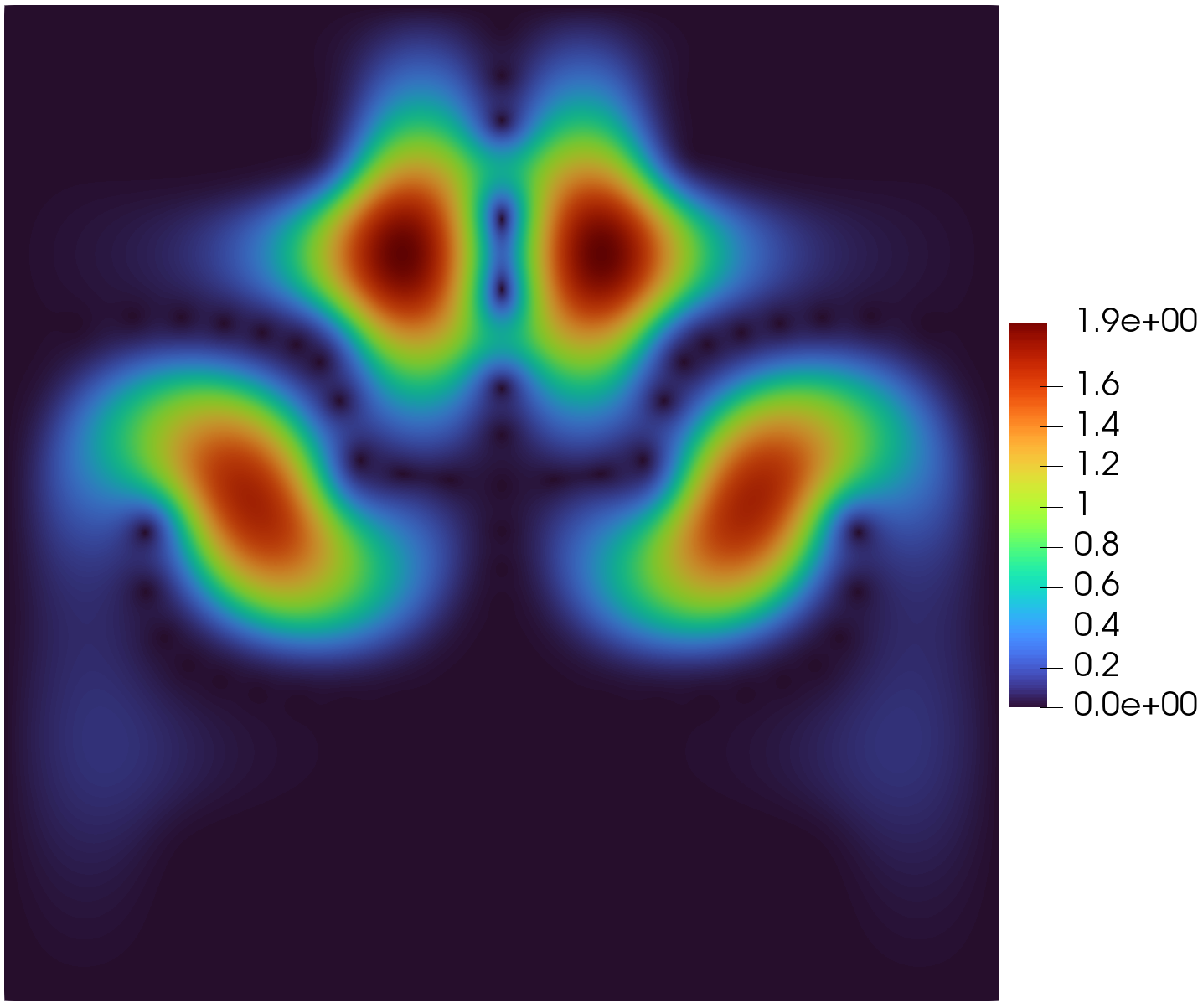}}
         \quad
         \subfloat[{ $\lambda\approx 89.936$.}]
 	{\includegraphics[width=0.30\textwidth]{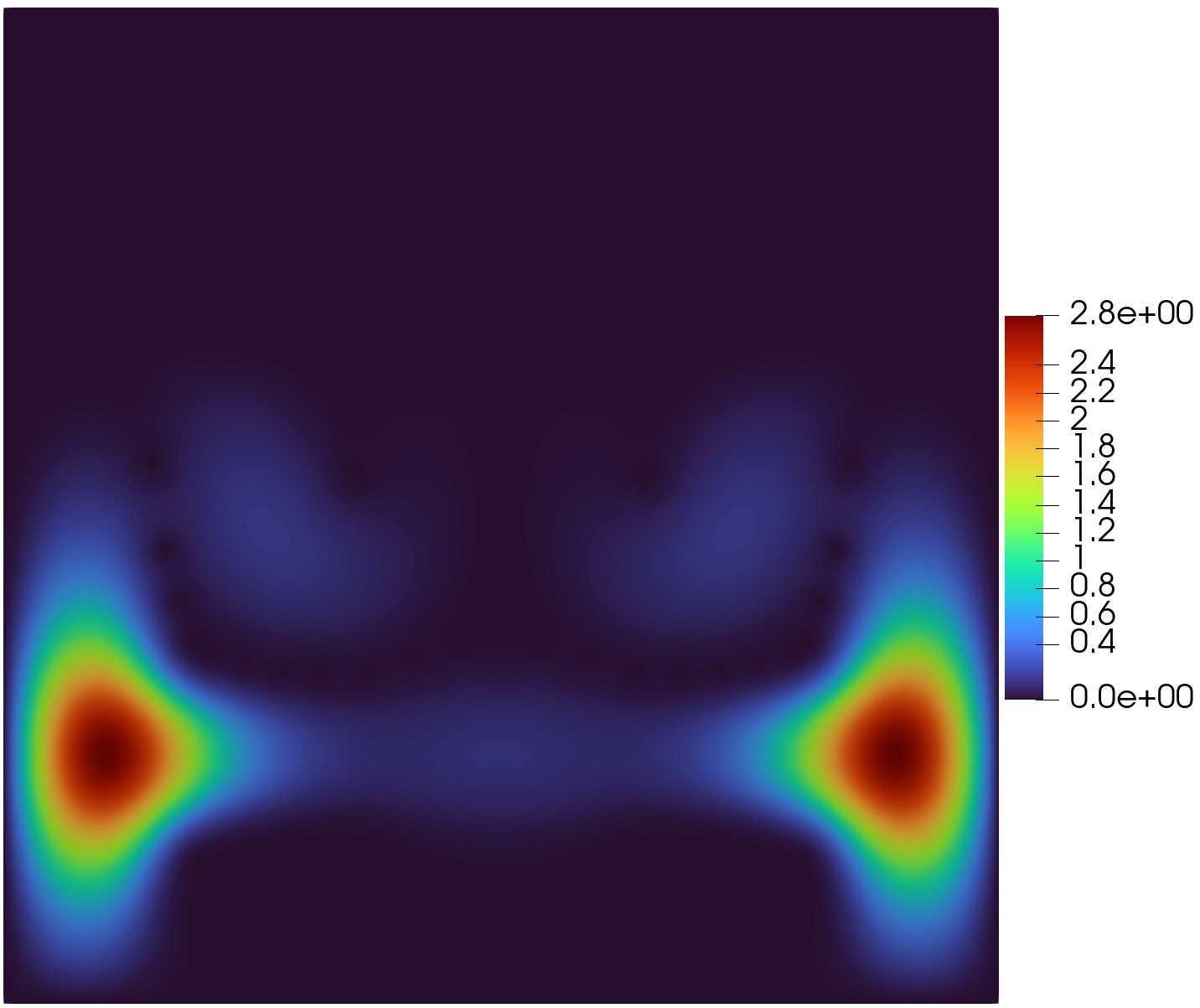}}
        \quad
  	\subfloat[{$\lambda\approx 89.949$.}]
 	{\includegraphics[width=0.30\textwidth]{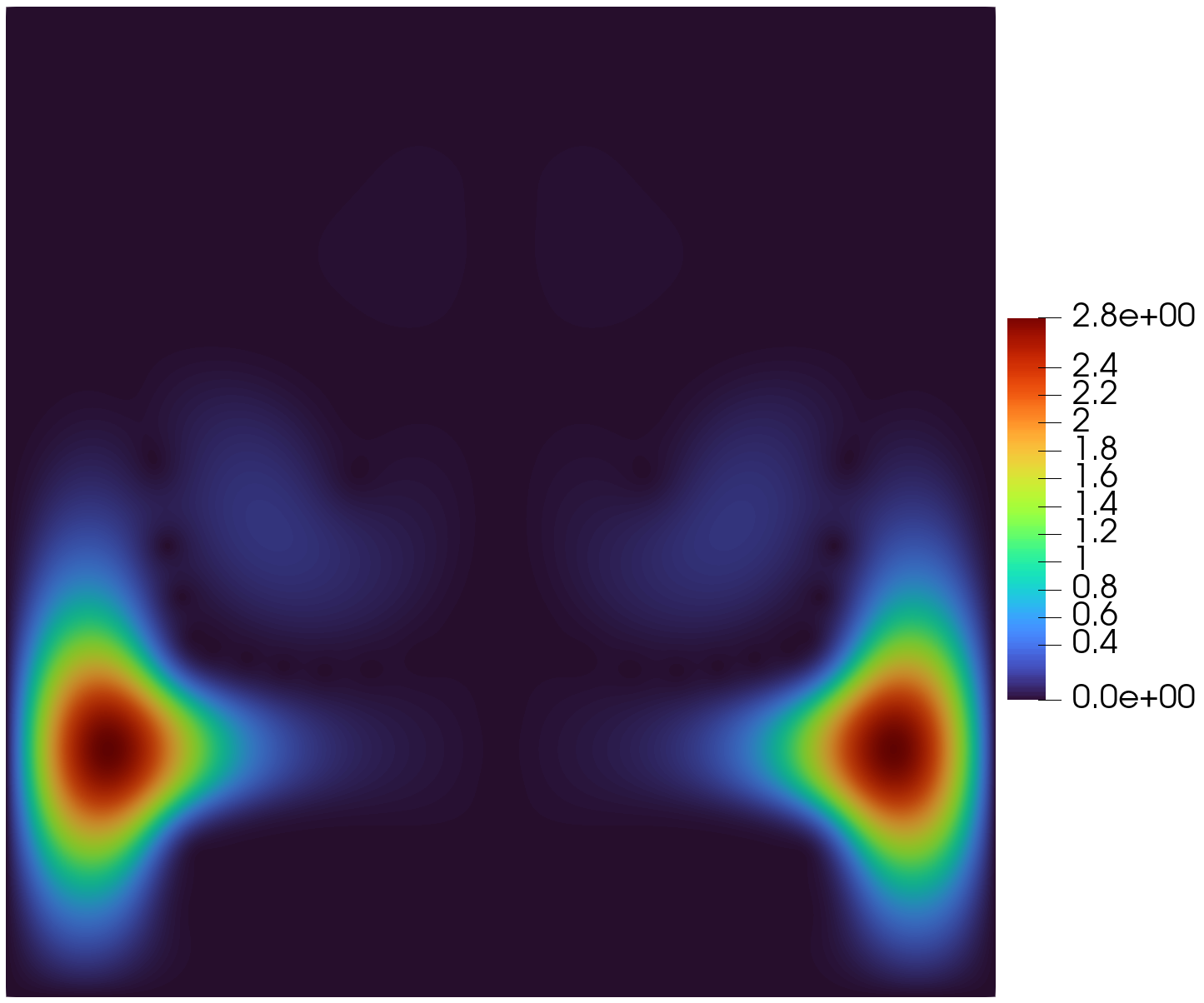}}\\
        \subfloat[{$\lambda\approx 99.499$.}]
	{\includegraphics[width=0.30\textwidth]{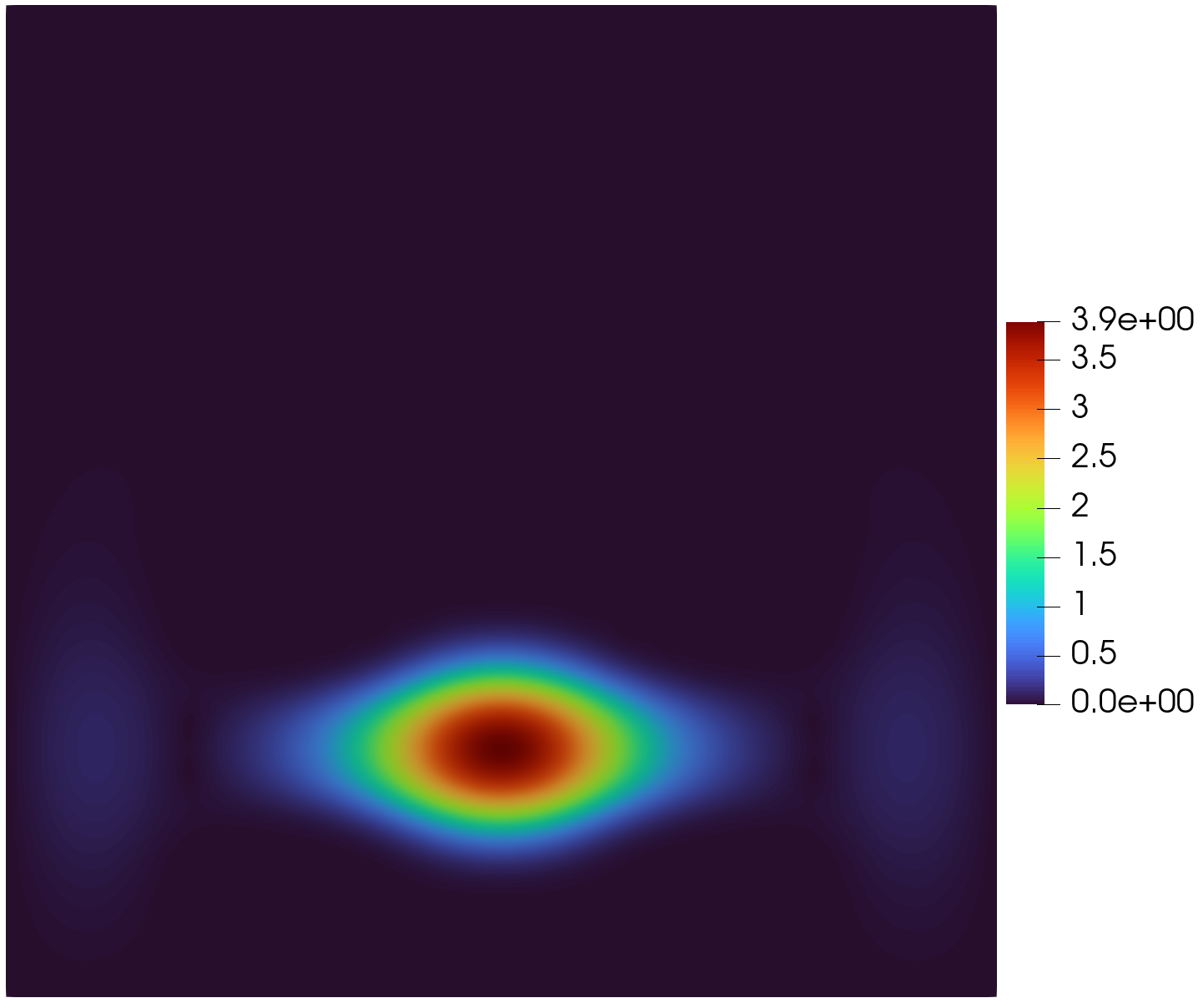}}
         \quad
         \subfloat[{$\lambda\approx 111.91$.}]
 	{\includegraphics[width=0.30\textwidth]{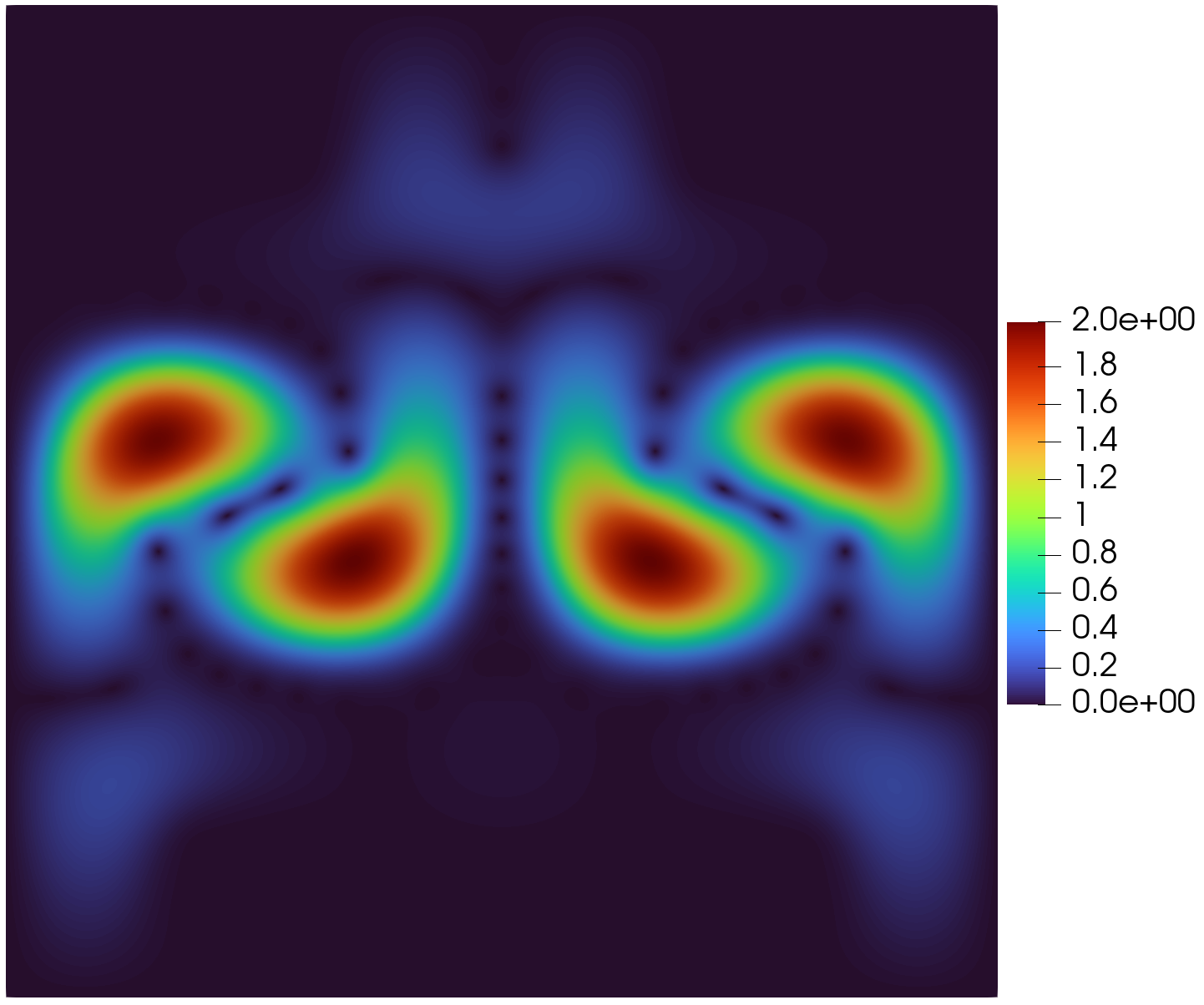}}
        \quad
  	\subfloat[{$\lambda\approx 111.95$.}]
 	{\includegraphics[width=0.30\textwidth]{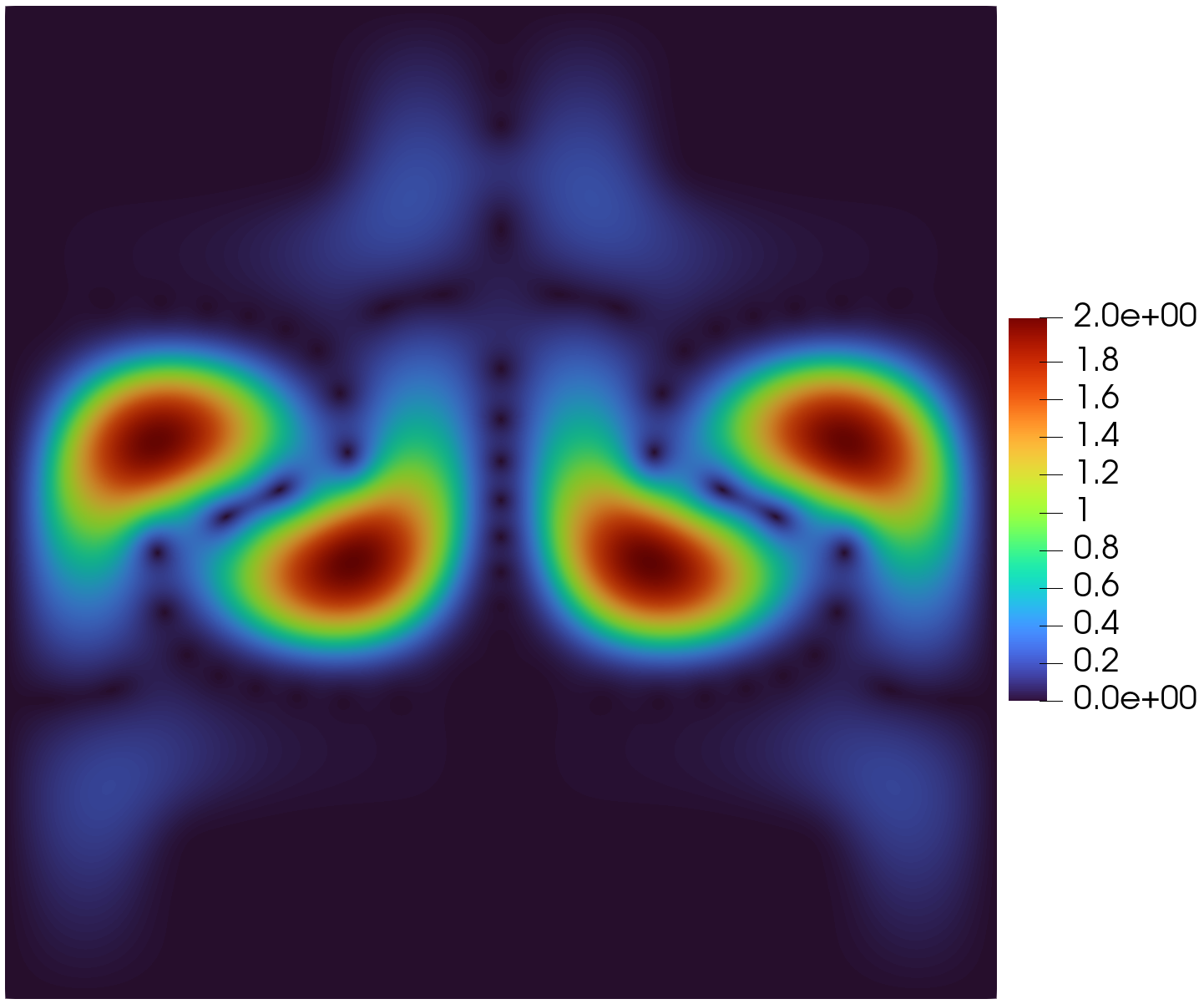}}        
        \caption{\label{Example5Fig} $A$, $|\curl A|$ and $|\psi|$
  for smallest nine eigenvalues. $A$ given by Example~\ref{Example4}, with $a=50$, see \S \ref{s:numerics}.}
\end{figure}

\subsection{Numerical Illustrations}\label{s:numerics}
We illustrate the heuristic that eigenfunctions of $H(A)$ corresponding to small eigenvalues
tend to localize in regions where the curl of $A$ is relatively small. We show a few such examples on the domain $\Omega=(-1,1)^2$.
Numerical experiments are conducted using Pythonic
FEAST~\cite{Gopalakrishnan2017}, which builds
on the general purpose finite element software package
NGSolve~\cite{Schoeberl2014,Schoeberl2021}.
Pythonic FEAST employs the FEAST algorithm
(cf.~\cite{Gopalakrishnan2019,Gopalakrishnan2020,Polizzi2009,Tang2014}) for solving
eigenvalue problems.  Modifications of the base algorithm allow for
operators of the type considered here.  The finite element spaces we employ
consist of continuous, piecewise cubic polynomials on regular
triangulations of $\Omega$ having characteristic edge length $h=0.01$.

\begin{example}[A simple polynomial vector field]\label{Simple}
Let $A =-a(x^2+y^2,x^2-y^2)$ with $a$ being a real parameter.  Then 
\begin{align*}
\curl A&=-2a(x-y)~.
\end{align*}
A stream plot of $A$, together with plots of $|\curl A|$ and
the modulus of a few eigenvectors are given in Figure~\ref{SimpleFig}, for the choice
$a=1000$.
The real and imaginary parts of $\psi$ for $\lambda\approx 120.52$ are
shown in Figure~\ref{SimpleFigB}, and illustrate how even the ground
state can be oscillatory in its components.
\end{example}

A recurring theme of the paper is that eigenfunctions localize in places where the vector field behaves approximately like a conservative vector field over short scales given by the wavelength. This naturally suggests that for typical vector fields $A(x,y)$ there is a predilection to localize in regions where $\|A(x,y)\|$ is small (relative to its size in other regions). This can be observed in practice. In order to be able to better observe other more subtle effects (such as the impact of the curl), we consider vector fields of the form 
$$  A=-a(\cos( f(x,y))\,,\,  \sin( f(x,y))),$$
where $a > 0$ is some constant and $f(x,y)$ is an arbitrary real-valued function. We have that $\|A\| = a$ is of constant size throughout.  

\begin{example}[A vector field with constant modulus]\label{Example3}
 We take the vector field $A=-a(\cos f(x,y)\,,\,
  \sin f(x,y))$ with $f(x,y)=5\pi\sin(x^2+y^2)$ and $a=50$.
  we plot $A$ and $|\curl A|$ together with the modulus of
  several eigenvectors, in Figure~\ref{Example3Fig}.  
\end{example}

\begin{example}[Another vector field with constant modulus]\label{Example4}
  We use the same type of vector field as in Example~\ref{Example3},
  but with $f(x,y)=\pi\sin(\pi x)\cos(\pi y)$ and $a=50$.  We plot $A$ and $|\curl A|$ together with the modulus of
  several eigenvectors in Figure~\ref{Example5Fig}.
\end{example}

\subsection{The full magnetic Schr\"odinger operator.} Although not emphasized in the paper, our main arguments extend to the full magnetic Schr\"odinger operator
$$ H(A,V)\psi = \left(- i\nabla - A(x)\right)^2 \psi(x) + V(x)\psi(x).$$ 
This operator also admits a Feynman-Kac formula (see \cite{broderix})
$$ \left[e^{-t H(A,V)} \phi\right](x) = \mathbb{E}_{\omega}\left( e^{-S_{t}(A,V|\omega)} \chi_{\Omega}(\omega,t) \phi(\omega(t)) \right),$$
where the expectation is taken over all Brownian motion $\omega$ started at $x$, and where 
$$\chi_{\Omega}(\omega,t)  = \begin{cases} 1 \qquad &\mbox{if}~\forall ~0\leq s \leq t: \omega(s) \in \Omega \\ 0 \qquad &\mbox{otherwise} \end{cases}$$
measures whether the Brownian motion has left the domain. The quantity $S_{t}(A,V|\omega)$ is defined via the formula
\begin{align*}
 S_{t}(A,V|\omega) &= i \int_0^{t} A(\omega(s)) \cdot d\omega(s) + \frac{i}{2} \int_0^{t} (\nabla \cdot A)(\omega(s)) ds \\
 &+ \int_0^{t} V(\omega(s)) ds.
\end{align*}
The same formula applies in our setting where it simplifies: $\nabla \cdot A$ vanishes through the Helmholtz decomposition and $V \equiv 0$. The general case admits a similar inequality as the main Theorem. However, there is a significant increase in complexity coming from the additional interplay between the vector field $A$ and the potential $V$. Nonetheless, we emphasize that any type of numerical method based on this path integral localization for magnetic Laplacians may admit, via this more general Feynman-Kac formula, a natural extension to the full magnetic Schr\"odinger operator; we consider this a promising avenue for future work.

\subsection{Related results}  To the best of our knowledge, these results are these first of their type. However, there are a number of existing results in the literature that share philosophical similarities. We especially emphasize the results of B. Poggi \cite{poggi} and Z. Shen \cite{shen}: in both papers, the curl of the magnetic field plays a significant role. While the curl also arises in our approach, it does so implicitly as a first order approximation -- in particular, we note that the nonlocal nature of our formulation is consistent with the Aharonov-Bohm effect.  There is also a rich semiclassical literature concerned with the magnetic Laplacian; we refer to Dimassi--Sj\"ostrand \cite{dim}, Helffer \cite{helffer}, Helffer-Sj\"ostrand \cite{helff2, helff3}. In particular, the curl appears naturally in the semiclassical limit \cite{ helffer-mohamed, helffer-mohamed2, helff4, mat, montgomery}.

\section{Proof of the Theorem}



\begin{proof}[Proof of the Theorem] We are given the operator
$$ H(A) \psi =  \left(- i\nabla - A(x)\right)^2 \psi~,$$
and consider eigenfunctions subject to Dirichlet boundary conditions $\psi\big|_{\partial \Omega} = 0$.
$H(A)$ is a self-adjoint operator with a real spectrum. We note that
$$ \left\langle H(A) \psi, \psi \right\rangle = \int_{\Omega} \left| \left(- i\nabla - A(x)\right) \psi \right|^2 dx \geq 0~,$$
so all eigenvalues are nonnegative. Under minimal assumptions on $A$ there is a Feynman-Kac formula (see Broderix, Hundertmark, Leschke \cite{broderix}).
Such assumptions are that $\nabla\cdot A$ and $A \cdot A$ are in the Kato class, so all differentiable vector fields are allowed.
This allows to rewrite the evolution operator as
$$ \left[e^{-t H(A)} \psi\right](x) = \mathbb{E}\left( e^{-S_t(A|\omega)} \chi_{\Omega}(\omega,t) \psi(\omega(t)) \right),$$
where the expectation is taken over all Brownian motion started at $x$.  Here, 
$$\chi_{\Omega}(\omega,t)  = \begin{cases} 1 \qquad &\mbox{if}~\forall ~0\leq s \leq t: \omega(s) \in \Omega \\ 0 \qquad &\mbox{otherwise} \end{cases}$$
measures whether the Brownian motion has left the domain, and 
\begin{align*}
 S_t(A|\omega) &= i \int_0^t A(\omega(s)) \cdot d\omega(s) + \frac{i}{2} \int_0^t (\nabla \cdot A)(\omega(s)) ds. 
\end{align*}
This representation will be very useful: since $\psi$ is an eigenfunction, the evolution operator is very simple and
$$  \left[e^{-t H(A)} \psi\right](x) = e^{-\lambda t} \psi(x).$$
At this point, we perform a Helmholtz decomposition
\begin{align*}
 A(x) = \nabla \phi + F(x) \qquad \mbox{where} \quad \phi \in C^1(\Omega) \quad ~\mbox{and} ~\div(F) = 0
\end{align*} 
and apply Lemma \ref{lem:nice} to instead consider the operator $ H(F) \psi_2 =  \left(- i\nabla - F(x)\right)^2 \psi_2$. This operator has an eigenfunction $\psi_2 = e^{- i \phi} \psi$ corresponding to the same eigenvalue $\lambda$. In particular, $|\psi(x)| = |\psi_2(x)|$ and they both localize in the same spot $x_0 \in \Omega$.
The stochastic integral for this divergence-free vector field simplifies to
\begin{align*}
 S_t(F|\omega) &= i \int_0^t F(\omega(s)) \cdot d\omega(s) 
\end{align*}
Note that this expression is purely imaginary because the integral is always real. At this point, the representation formula has been simplified to
$$ \psi_2(x) = e^{\lambda t}  \left[e^{-t H(F)} \psi_2\right](x) = e^{\lambda t} \cdot  \mathbb{E}\left( e^{-S_t(F|\omega)} \chi_{\Omega}(\omega,t)\psi_2(\omega(t)) \right).$$
We will now apply this formula at the point $x_0 \in \Omega$ where the eigenfunction attains its maximum, $|\psi_2(x_0) | = |\psi(x_0)| = \|\psi\|_{L^{\infty}} = \|\psi_2\|_{L^{\infty}}$. 
We condition the expectation on $\omega(t) = y$. We have complete control over the likelihood of the probability distribution of Brownian particles starting at $x_0$ and running for $t$ units of time without ever leaving the domain: this is the heat kernel, $y \rightarrow p_t(x_0, y)$. Thus,
$$ \mathbb{E} \left( e^{-S_t(F|\omega)} \chi_{\Omega}(\omega,t)\psi_2(\omega(t)) \right) = \int_{\Omega} \left(  \mathbb{E}_{\omega(0)=x_0,\omega(t) =y} e^{-S_t(F|\omega)}  \right)  \psi_2(y) p_t(x_0,y) dy.$$
This identity will be useful in the proof of Corollary \ref{cor:3}. Altogether, recalling that $|\psi_2(x_0)| = \| \psi\|_{L^{\infty}}$, we have
$$ e^{-\lambda t} \| \psi\|_{L^{\infty}} = \left| \mathbb{E}_{}\left( e^{-S_t(F|\omega)} \chi_{\Omega}(\omega,t)\psi_2(\omega(t)) \right)\right|~,$$
and we bound its absolute value from above by
\begin{align*}
 e^{-\lambda t} \| \psi\|_{L^{\infty}} &= \left|   \int_{\Omega} \left(  \mathbb{E}_{\omega(0)=x_0,\omega(t) =y} e^{-S_t(F|\omega)}  \right)  \psi_2(y) p_t(x_0,y) dy \right| \\
 &\leq \|\psi_2\|_{L^{\infty}} \int_{\Omega}  \left| \left[ \mathbb{E}_{\omega(0)=x_0,\omega(t) =y} e^{-S_t(F|\omega)} \right] \right|  p_t(x_0,y) dy.
\end{align*}
This implies
$$ \int_{\Omega}  \left| \left[  \mathbb{E}_{\omega(0)=x_0,\omega(t) =y} e^{-S_t(F|\omega)} \right] \right|  p_t(x_0,y) dy \geq  e^{-\lambda t}.$$
As a final step, we can remove the dependency on the heat kernel by using the comparison bound with the Euclidean heat kernel,
$$ 0 \leq p_t(x_0,y) \leq  \frac{1}{(4 \pi t)^{d/2}} e^{- \frac{\|x_0-y\|^2}{4t}}~,$$
which is valid for any $\Omega \subset \mathbb{R}^d$.
This proves the main result. \end{proof}

We note that the last step of the argument, replacing the heat kernel by the free heat kernel is extremely accurate for small values of $t$, where small means $t \ll d(x,\partial \Omega)^2$ and `extremely accurate' means that the errors are exponentially small.

\section{Proof of Corollary \ref{cor:1}}

\begin{lemma} \label{lem:mini}
Let $X$ be a real-valued random variable and suppose that
$$ \sup_{z \in \mathbb{T}} \mathbb{P}\left(  \left| X \mod 2\pi  - z \right| \leq \frac{1}{100} \right) \geq \frac{99}{100}.$$
Then, for some universal constant $0 \leq c_1 < 1$,
$$ \left| \mathbb{E} \exp(i X) \right| \leq c_1.$$
\end{lemma}
\begin{proof}
This follows from the convexity of $\mathbb{S}^1$.
\end{proof}

\begin{proof}
We negate the statement and let $c \rightarrow 0$. Negating the statement means that for any $c>0$ there is an example of an eigenfunction $H(A) \psi = \lambda \psi$ in some domain that is localized at $x_0$ and where, for $t = c/\lambda$, the measure of points in $\sqrt{c/\lambda}-$neighborhood of $x_0$ that is near-deterministic is only $9/10-$th of the measure. The proof of the main result implies, for $t=c/\lambda$,
$$ e^{-c} \leq \int_{\Omega}  \left| \left[ \mathbb{E}_{\omega(t) =y} e^{-S_t(F|\omega)} \right] \right|  p_t(x,y) dy.$$
We introduce the set 
$$\Omega_1 = \left\{y \in \Omega \cap B_{\sqrt{t}}(x_0): y~\mbox{not near-deterministic} \right\}$$
and decompose $\Omega = \Omega_1 \cup (\Omega \setminus \Omega_1)$. Then, appealing to Lemma \ref{lem:mini},
\begin{align*}
\int_{\Omega}  \left| \left[ \mathbb{E}_{\omega(t) =y} e^{-S_t(F|\omega)} \right] \right|  p_t(x,y) dy &\leq \int_{\Omega_1}  \left| \left[ \mathbb{E}_{\omega(t) =y} e^{-S_t(F|\omega)} \right] \right|  p_t(x,y) dy \\
&+\int_{\Omega \setminus \Omega_1}  \left| \left[ \mathbb{E}_{\omega(t) =y} e^{-S_t(F|\omega)} \right] \right|  p_t(x,y) dy\\
&\leq  c_1 \int_{\Omega_1}  p_t(x,y) dy +  \int_{\Omega \setminus \Omega_1}    p_t(x,y) dy.
\end{align*}
Since $\Omega_1$ contains a positive proportion of the measure in the $\sqrt{t}$-neighborhood and $p_t(x, \cdot)$ contains a positive proportion of measure in the same neighborhood, we have
$$ c_1 \int_{\Omega_1}  p_t(x,y) dy +  \int_{\Omega \setminus \Omega_1}    p_t(x,y) dy \leq 1 - c_2$$
for some absolute constant $c_2$. This then leads to a contradiction when $c \rightarrow 0$.
\end{proof}

\textit{Remark.} This argument clearly has some wiggle room. One could, for example, let one of the parameters go to 0 (or 1) at a certain rate depending on $c$. However, since  Corollary \ref{cor:1} is mainly intended to be an illustrative example of the more important general underlying principle, we leave such variations to the interested reader.

\section{Proof of Corollary \ref{cor:2}}
\begin{proof}
  Using Lemma \ref{lem:nice} we may assume without loss of generality that we have conjugated $H(A)$ by some $e^{-i\phi}$ to remove the conservative part of $A$, i.e. that $\nabla \cdot A\equiv 0$.  We start with a Taylor expansion of the vector field,
  $A = A_{\mbox{\tiny lin}} + A_{\mbox{\tiny nonlin}}$, followed by a corresponding Taylor expansion of the path integral.  Recall from~\eqref{LinearSplitting}, that $A_{\mbox{\tiny lin}}=A_1+A_2$, where
  \begin{align*}
    A_1(x)=A(x_0)+\frac{1}{2}\left(J(x_0)+J(x_0)^T\right)(x-x_0)\;,\;
    A_2(x)=\frac{\curl A(x_0)}{2}\,R(x-x_0)~.
  \end{align*}
  It is clear that $A_1=\nabla f$, where
  \begin{align*}
    f(x)=A(x_0)\cdot(x-x_0)+\frac{1}{2}\, (x-x_0)^T J(x_0) (x-x_0)~.
  \end{align*}
  It follows that
  \begin{align*}
    \nabla\cdot A(x_0)&=\nabla\cdot A_{\mbox{\tiny lin}}(x)=\nabla\cdot A_{1}(x)=\Delta f(x)~,\\
    \curl A(x_0)&=\curl A_{\mbox{\tiny lin}}(x)=\curl A_{2}(x)~.
  \end{align*}
  By assumption, $\Delta f=\nabla\cdot A(x_0)=0$, so we have, by appealing to It\^o's lemma, 
$$ \int_0^t A_1 \cdot d\omega(s) = f(\omega(t)) - \tfrac{1}{2}\int_0^t \Delta f(\omega(s))ds = f(\omega(t))~. $$
It follows that 
$$ \int_0^t A_{\mbox{\tiny lin}}(x,y)  \cdot d\omega(s) = f(\omega(t)) + \frac{\curl A(x_0)}{2}\int_0^t R(x-x_0) \cdot d\omega(s)~,$$
where the first term, $f(\omega(t))$, is completely deterministic (since $\omega(t) = y$). 
It remains to understand the random variable
$$ \int_0^t R(x-x_0) \cdot d\omega(s)~,$$
conditioned on $\omega(0) = x_0$ and $\omega(t) = y$.  This random variable is not deterministic, as it will depend on the actual path the Brownian motion takes.
At this point, we assume, without loss of generality, that $x_0=0$.
We first note that both Brownian motion and the vector field are invariant under rotation.  Therefore, the random variable can only depend on $\| \omega(t) \| = \|y\|$  and $t$. There is an additional scaling symmetry. Note that, for any parameter $\alpha > 0$, the Brownian motion satisfies 
$$ \omega(\alpha\,t)  \equiv \sqrt{\alpha} \cdot \omega(t)$$ 
in the sense of both random processes being identical.
This leads to a parabolic scaling. The likelihood of a fixed Brownian particle traveling along a fixed path from $\omega(0) = 0$ to $\omega(t) = y$ is the same as the likelihood of the rescaled particle $\omega(0) = 0$ traveling to $\omega(1) = y t^{-1/2}$. Therefore, after a change of variables,
$$ \int_0^t Rx\cdot d\omega(s) = t \int_0^1 Rx\cdot d\omega(t \cdot s) = t^2 \int_0^1 Rx \cdot d\omega(s).$$
However, for $t=1$ and $\|y\| \sim 1$, this is simply a random variable with some deterministic mean value depending only on the endpoint, and with some nonzero standard deviation spread over an interval of size $\sim 1$. Therefore, by scaling,
$$ \frac{\curl A(x_0)}{2}\int_0^t Rx \cdot d\omega(s)~ \quad \mbox{is a random variable spread over} \quad
\frac{|\curl A(x_0)|}{2}\, t^2~.$$
Suppose now that $(\lambda,\psi)$ is an eigenpair of $H(A)$, with $|\psi(x_0)|=\|\psi\|_{L^\infty}$.
Then
$$ \int_{\Omega}  \left| \left[ \mathbb{E} ~e^{-S_t(A|\omega)} \big| \omega(0) = x_0 \wedge \omega(t) =y \right] \right|  p_t(x,y) dy \geq  e^{-\lambda t}.$$
For $t = 0.01/\lambda$, we deduce that
$$ \int_{\Omega}  \left| \left[ \mathbb{E} ~e^{-S_t(A|\omega)} \big| \omega(0) = x_0 \wedge \omega(t) =y \right] \right|  p_t(x,y) dy \geq  0.95.$$
Ignoring the non-linear part of the vector field, this inequality by itself requires that the path integral is highly concentrated over most points and, for some universal constant $c>0$, 
$$ \frac{| \curl A(x_0)|}{2} \, \frac{0.1^2}{\lambda^2} \leq c.$$
Rescaling and renaming the constant $c$, we deduce
$$  \left| \curl A(x_0) \right| \leq c \cdot \lambda^2.$$
We conclude by noting that, in order for this argument to be correct, we require that the higher-order contributions are locally small and
$$ \left| \int_0^t A_{\mbox{\tiny nonlin}}(x,y)  \cdot d\omega(s) \right| \ll   \left| (\curl A)(x_0) \right| \cdot \frac{1}{\lambda^2}.$$
 \end{proof}

\section{Proof of Corollary \ref{cor:3}}
\begin{proof}
Going through the proof of the main result, we can skip the step of bounding the heat kernel in terms of the Gaussian and will arrive at the inequality
$$\bigintss_{\Omega} \left| \mathbb{E}_{\omega(0) = x, \omega(t) = y}  \exp\left( i  \int_{0}^{t} F \cdot d\omega(s) \right) \right|   p_t(x,y) dy \geq e^{-\lambda \cdot t} \frac{|\psi(x)|}{\| \psi\|_{L^{\infty}}}.$$
We integrate both sides of the inequality over $\mathbb{R}_{\geq 0}$ with respect to $t$ to get
$$ \frac{1}{\lambda}  \frac{|\psi(x)|}{\| \psi\|_{L^{\infty}}} \leq  \int_0^{\infty}  \int_{\Omega} \left| \mathbb{E}_{\omega(0) = x, \omega(t) = y}  \exp\left( i  \int_{0}^{t} F \cdot d\omega(s) \right) \right|   p_t(x,y) dy dt.$$
If we denote by $\mathbb{E}$ the expected value of the stochastic integral and use the Cauchy-Schwarz inequality, we obtain
$$ \frac{1}{\lambda}  \frac{|\psi(x)|}{\| \psi\|_{L^{\infty}}} \leq \left(  \int_0^{\infty}  \int_{\Omega} |\mathbb{E}|^2 \cdot p_t(x,y) dy \,dt\right)^{1/2}  \left(  \int_0^{\infty}  \int_{\Omega} p_t(x,y) dy \,dt\right)^{1/2}.$$
The second integral has a simple closed-form expression: recall that
$$ p_t(x,y) = \sum_{k=1}^{n} e^{-\lambda_k t} \phi_k(x) \phi_k(y)~,$$
from which we deduce that
$$ \int_{\Omega} \int_{0}^{\infty}    p_t(x,y) dt dy= \sum_{k=1}^{n} \frac{1}{\lambda_k} \phi_k(x) \left( \int_{\Omega} \phi_k(y) dy \right).$$
Now we solve the PDE $-\Delta v = 1$ by expanding the solution $v$ into the eigenfunctions of the Laplacian $-\Delta \phi_k = \lambda_k \phi_k$. Taking inner products on both sides leads to
$$\int_{\Omega} \phi_k dx = \left\langle 1, \phi_k \right\rangle =  \left\langle -\Delta v, \phi_k \right\rangle =   \left\langle v,  -\Delta \phi_k \right\rangle = \lambda_k  \left\langle v,  \phi_k \right\rangle ,$$
from which we deduce that
$$  \int_{\Omega} \int_{0}^{\infty}    p_t(x,y) dt dy = v(x) \qquad \mbox{where} \qquad -\Delta v(x) = 1$$
with Dirichlet boundary conditions.
\end{proof}


\end{document}